\documentclass{hha}

\usepackage{subfig}
\RequirePackage[OT1]{fontenc}
\usepackage{amsfonts,latexsym,amssymb}
\usepackage{graphics,graphicx}
\usepackage{xcolor}
\usepackage{bbm}
\usepackage{enumitem}
\usepackage{epstopdf}

\DeclareMathOperator{\supp}{supp}

\DeclareMathOperator{\diam}{diam}
\DeclareMathOperator{\vol}{Vol}

\DeclareMathOperator{\conv}{conv}

\def\N{\mathbb{N}}

\def\R{\mathbb{R}}

\def\definedas{:=}

\def\convinltwo{\stackrel{L_2}{ \to}}
\def\cC{\mathcal{C}}

\def\cI{\mathcal{I}}
\def\cL{\mathcal{L}}
\def\cM{\mathcal{M}}
\def\cN{\mathcal{N}}
\def\cP{\mathcal{P}}
\def\cB{\mathcal{B}}
\def\cS{\mathcal{S}}
\def\cX{\mathcal{X}}
\def\cY{\mathcal{Y}}

\def\oconv{\conv^{\circ}}

\newcommand{\E}{\mathbb{E}} 

\newcommand{\given}{\;|\;}
\newcommand{\mean}[1] {\E\left\{{#1}\right\}}
\newcommand{\cmean}[2] {\E\left\{#1\given #2\right\}}
\newcommand{\ind}{\boldsymbol{\mathbbm{1}}} 

\newcommand{\var}[1]{\mathrm{Var}\param{{#1}}}

\newcommand{\set}[1]{\left\{#1\right\}}

\newcommand{\norm}[1]{\left\|#1\right\|}
\newcommand{\param}[1]{\left(#1\right)}
\newcommand{\abs}[1] {\left| {#1}\right|}

\newcommand{\prob}[1]{\mathbb{P}\left(#1\right)}
\newcommand{\cprob}[2]{\mathbb{P}\left(#1\given #2\right)} 

\newcommand{\eps}{\epsilon}
\newcommand{\bz}{\mathbf{z}}
\newcommand{\by}{\mathbf{y}}
\newcommand{\bx}{\mathbf{x}}
\newcommand{\bi}{{\bf{{i}}}}
\newcommand{\bj}{\mathbf{j}}

\newtheorem{lem}{Lemma}[section]
\newtheorem{thm}[lem]{Theorem}
\newtheorem{prop}[lem]{Proposition}
\newtheorem{cor}[lem]{Corollary}

\theoremstyle{definition}
\newtheorem{defn}[lem]{Definition}

\newcommand{\cech}{\v{C}ech }

\newcommand{\Rd}{\R^d}
\newcommand{\iid}{\mathrm{i.i.d.}}
\def\Nk{N_{k,n}}
\def\Sk{S_{k,n}}
\def\Nkg{\Nk^{(g)}}

\def\Nkt{\widehat{N}_{k,n}}

\newcommand{\Nkts}[1]{{N}_{k,n}^{{#1}}}

\def\bk{\beta_{k,n}}
\def\bkm{\beta_{k-1,n}}

\newcommand{\rnd}{r_n^d}
\newcommand{\rndk}{r_n^{dk}}
\newcommand{\factor}{n^{k+1} \rndk}
\newcommand{\ninf}{n\to\infty}
\newcommand{\pois}[1]{\mathrm{Poisson}\param{{#1}}}
\newcommand{\grn}{g_{r_n}}
\newcommand{\hrn}{h_{r_n}}
\newcommand{\fmax}{f_{\max}}
\newcommand{\fmin}{f_{\min}}
\newcommand{\limninf}{\lim_{\ninf}}
\newcommand{\qin}{Q_{i,n}}

\newcommand{\dtv}[2]{d_{\mathrm{TV}}\param{{#1},{#2}}}
\newcommand{\bs}{\backslash}
\newcommand{\non}{\nonumber}

\def\CC{\check{C}}
\def\const {c^\star}

\numberwithin{equation}{section}

\def\:{:\,}

\newcommand{\bc}{\begin{center}}
\newcommand{\ec}{\end{center}}
\newcommand{\beq}{\begin{eqnarray}}
\newcommand{\eeq}{\end{eqnarray}}
\newcommand{\beqq}{\begin{eqnarray*}}
\newcommand{\eeqq}{\end{eqnarray*}}

\begin{document}

\title{Distance Functions, Critical Points, and  the Topology of Random \cech Complexes}

\shorttitle{Distance Functions}


\author{Omer Bobrowski}             
\email{omer@math.duke.edu}
\address{Department of Mathematics,
         Duke University,
         120 Science Dr.
         Durham, NC, 27708
         USA}
\thanks{O.B. was supported in part by the Adams Fellowship Program of the Israel Academy of Sciences and Humanities and FP7-ICT-318493-STREP, TOPOSYS.}

\author{Robert J. Adler}             
\email{robert@ee.technion.ac.il}
\address{Department of Electrical Engineering,
         Technion - Israel Institute of Technology,
         Haifa,32000
         Israel}

\thanks{R.A. was supported in part by AFOSR FA8655-11-1-3039 and ERC  2012 Advanced Grant 20120216, URSAT.}


\classification{60D05, 60F05, 60G55, 55U10, 58K05}

\keywords{Distance function, critical points, Morse index, \cech complex, Poisson process, central limit theorem, Betti numbers}

\begin{abstract}
For  a finite set of points $\cP$ in $\R^d$, the function
$d_{\cP}\:\R^d\to\R^+$ measures Euclidean distance to the set $\cP$. We study the number of critical points of $d_{\cP}$ when $\cP$ is
a Poisson process.
In particular, we study
the limit behavior of $N_k$ -- the number of critical points of $d_{\cP}$
with Morse index $k$ -- as the
density of points grows.
 We present explicit computations for the normalized, limiting,
expectations and variances of the $N_k$, as well as distributional limit theorems.  We link these results
to recent  results in \cite{kahle_random_2011, kahle_limit_2010} in which
 the Betti numbers of the random \cech complex based on $\cP$ were studied.
\end{abstract}

\maketitle


\section{Introduction}
For a finite set $\cP$   of points in $\R^d$, of size $|\cP |$, let $d_{\cP}\:\R^d\to\R^+$ be the distance function for $\cP$, so that
\beq
\label{dfunction:def}
d_{\cP}(x) \definedas \min_{p\in\cP}\|{x-p}\|_2, \quad x\in\R^d,
\eeq
where $\|\cdot\|_2$ denotes the Euclidean distance.

The main results of this paper provide considerable information about the
asymptotic (in $|\cP |$)
behavior of the critical points (defined below)  of $d_{\cP}$
when  $\cP$ is random. While the critical points
 are, by themselves,
 intrinsically
interesting, knowledge of their behavior also has immediate implications
(via Morse theory) to the study of the topology of \cech complexes built
over random point sets.

Throughout, we shall concentrate on the situation in which the points in $\cP$ are those of a non-homogeneous
Possion process with intensity $\lambda_n=nf$, where $f$ is a probability density on $\R^d$. The mean number of points is therefore $\mean{\abs{\cP}}=n$. Virtually identical results  hold when $\cP$ is made up of $n$ independent samples from $f$, and proofs in this situation can be found in the PhD thesis \cite{bobrowski2012thesis}.

Most of what we shall have to say will concentrate
on the distance function in neighborhoods of radius $r_n$ around $\cP$,
 when $n\to\infty$ and $r_n\to 0$. Our main results give expressions for the normalized, asymptotic, means and variances of $\Nk$ - the number of critical points with index $k$ appearing within distance $r_n$ from $\cP$,
along with various distributional limit results. The limit distributions are
of different kinds, and, depending on delicate relationships between
$d$, $k$, $r_n$ and $n$,  provide limits that may be Gaussian, Poisson, or
deterministic, while also exhibiting a range of critical phenomena.
Note that there are various notions of convergence used in probability theory, and so in Appendix \ref{sec:lim_dist} we provide  definitions of the notions that we need.
Our main results on critical points are
described in detail in Section  \ref{sec:results}. 
However, before stating the results, we first need to describe precisely how to define the critical points, along with their indices, for the distance function. The difficulty lies in the fact that the distance function is not everywhere
differentiable. We shall do this in the following section.

In Section \ref{sec:topology} we shall
discuss the relationship between $\Nk$ and the Betti numbers
of a special simplicial complex, the \cech complex, based on $\cP$.
The homology of the \cech complex is closely related to the neighborhood set,
or $r_n$-tube around $\cP$,
\beq
\label{ballunion}
\cB_n\ \definedas\ \bigcup_{p\in\cP}B_{r_n}(p),
\eeq
where $B_{r_n}(p)$ is the $d$-ball of radius $r_n$ around $p$.
What we shall see in Section   \ref{sec:topology} is that, if $r_n$ is
too small, then the individual balls in \eqref{ballunion} will generally fail to intersect, and the topology will be approximately that of a large number of disjoint points.
This is occasionally referred to as the ``dust" or ``sparse"  regime, although there does not yet seem to be a universally accepted term. If $r_n$ decays too slowly, then the balls will connect and the topology of
$\cB_n$ will be that of a single ball. At the (phase)  transition
$\cB_n$ will have a percolative-like structure, and so we call this
the {\it percolation phase}, which is also known as the ``thermodynamic limit". Each of these phases exhibits different limit
behavior, with even more subtle differences possible within phases
depending on interactions between parameters.

Translating our results about critical points into statements about the
(algebraic) topological structure of
$\cB_n$, as $n\to\infty$, will also allow us
 to compare them to other results currently in
the literature (primarily \cite{kahle_random_2011,kahle_limit_2010}). The one comment that we already make at this stage, however,
is that we can provide a much richer set of results for the asymptotic
behavior of numbers of critical points than is currently available for
the Betti numbers of these \cech complexes. Indeed, we can also provide some
topological results via critical points that are not yet available with a
direct topological approach. For example, we are able to compute properties of the Euler characteristic
$\chi_n$  of the complex, and can show (see Corollary \ref{cor:cech_ec} for details)
that there exist functions $\gamma_k$ such that
\beqq
\qquad \limninf n^{-1} \mean{\chi_n}=
\begin{cases}
1 & n\rnd \to 0,\\
1+\sum_{k=1}^d {(-1)^k \gamma_k(\lambda) }&    
n\rnd \to \lambda\in(0,\infty),\\
0 & n\rnd \to \infty.
 \end{cases}
\eeqq
Moreover, when $n\rnd\to \infty$ and $n\rnd \ge D^\star \log n$ for some  $D^\star$ (cf. Proposition \ref{prp:global_vs_local}), then $\mean{\chi_n}\to 1$.

The remainder of the paper contains the proofs of the results in
Sections  \ref{sec:results} and \ref{sec:topology}. These are
 organized in a number of sections and appendices
so as to make them as user friendly
as possible. Many of the proofs rely on techniques in the theory of random
geometric graphs as developed in \cite{penrose_random_2003}.

Finally, a few words on motivation. There is considerable current interest in
the study, from a topological, homological, point of view of random structures such as graphs and simplicial complexes. Some recent references are
\cite{aronshtam2010vanishing,babson2011fundamental,bobrowski2010euler,cohen2010homotopical,meshulam2009homological,pippenger2006topological} with two reviews,
from different aspects, in \cite{adler2010persistent} and \cite{ghrist_barcodes:_2008}. Many of these papers find their {\it raison d'\^etre} in essentially
statistical problems, in which data generates these structures. An important
example appears in  the papers
 \cite{niyogi_finding_2008, niyogi_topological_2010} which show that
the homology of an unknown manifold can be recovered, with high probability,
 by looking at the homology of the union of balls around  the points
of  random samples (or equivalently, at the homology of the \cech complex generated by the sampling
points on the manifold) with or without additional noise. The homological
theme of these papers, which considers manifolds as being `close' if their
homologies are the same,  seems
particularly promising for situations in which the manifold of interest is
embedded in a space of much higher dimension that itself; i.e.\ in
dimension reduction problems and in manifold learning.

The approach adopted in this paper shares the motivation of the others listed
above, but as  already noted,
by adopting a Morse theoretic point of view based on
 critical points of the distance function, obtains a more internally
complete theory.  Further, as mentioned above and shown later,  it often gives some information on 
 global topological invariants, such as Betti numbers. However, being based on critical points,  this approach is naturally limited
in its ability to reveal the full picture about the global topological invariants of random complexes.

\vskip0.1truein\noindent
{\it Acknowledgements.} We  thank Shmuel Weinberger for introducing us
to this problem, as well Matthew Strom Borman, Yuliy Baryshnikov, Matthew Kahle, and
Shmuel, for many useful discussions in the  earlier stages of our work.

\section{Critical Points of the Distance Function}
Critical points of smooth functions have been studied since the earliest days
of calculus, but took on significant additional importance following the
development of Morse theory (e.g.\ \cite{milnor1963morse,morse1969critical})
which tied them closely to the homologies of manifolds, a topic that
we shall discuss briefly in Section \ref{sec:topology}. At this point we
note that if  $\cM$ is a nice (closed, differentiable) $n$-dimensional manifold, and  $f\:\cM\to\R$  a nice (Morse) function, then a point $c$ is called a
 critical point if $\nabla f(c) = 0$. A non-degenerate critical point is
one for which the Hessian matrix $H_f(c)$ is non-singular. The Morse index $k\in\set{0,1,\ldots,n}$ of a non-degenerate critical point $c$ is then the number of negative eigenvalues of $H_f(c)$. These points, along with their indices, provide
one of the main links between differential and algebraic topology.

Classical Morse theory does not directly apply to the distance function
mainly because it is not everywhere differentiable. However, when the set $\cP$ is finite, one can still define a notion of non-degenerate critical points for the distance function $d_{\cP}$, as well as their Morse index. It turns out that, even in this case, knowledge of the critical points and their indices allows one to deduce
topological properties of the related \cech complexes. We shall see how to do this later in Section \ref{sec:topology}, but for
now we need some definitions.
Our arguments follow from
the results presented in \cite{gershkovich_morse_1997}. While the distance function served as the main motivation in \cite{gershkovich_morse_1997}, the results presented there are given in the more general context of `min-type' functions.
Here, we specialize those results to the case of the distance function.

Given a finite set  of points ${\cal P}\subset\R^d$, and defining the distance function
$d_{\cP}$  \eqref{dfunction:def},
we start with the local (and global) minima of $d_{\cP}$; viz.\ the points of
$\cP$ (where $d_{\cP} = 0$), and call these
critical points with index $0$. For higher indices,
 we have the following definition.

\begin{defn}\label{def:crit_pts}
A point $c\in\R^d$ is \emph{a critical point of  $d_{\cP}$ with index $1 \le k \le d$} if there exists a subset $\cY$ of $k+1$ points in $\cP$ such that:
\begin{enumerate}
\item $\forall y\in \cY\:  d_{\cP}(c) = \norm{c-y}_2 $, and, $\forall p\in \cP \backslash \cY$ we have $\norm{c-p}_2 > d_{\cP}(p)$.
\item The points in $\cY$ are in general position (i.e.\! the $k+1$ points
of $\cY$ do not lie in a $(k-1)$-dimensional affine space).
\item $c \in \oconv(\cY)$,
where $\oconv(\cY)$ is the interior of the convex hull of $\cY$ (an open $k$-simplex in this case).
\end{enumerate}
\end{defn}
The first condition implies that $d_{\cP} \equiv d_{\cY}$ in a small neighborhood of $c$.
The second condition implies that the points in $\cY$ lie on a unique $(k-1)$- dimensional sphere. We shall use the following notation:
\begin{align}
S(\cY) &= \textrm{The unique $(k-1)$-dimensional sphere containing $\cY$},\\
C(\cY) &= \textrm{The center of $S(\cY)$ in $\R^d$}, \label{eq:def_C}\\
R(\cY) &= \textrm{The radius of $S(\cY)$} , \label{eq:def_R}\\
B(\cY) &= \textrm{The open ball in $\R^d$ with radius $R(\cY)$ centered at $C(\cY)$},\label{eq:def_B}
\end{align}
Note that $S(\cY)$ is a $(k-1)$-dimensional sphere, whereas $B(\cY)$ is a $d$-dimensional ball. Obviously, $S(\cY) \subset B(\cY)$, but unless $k=d$, $S$ is
{\it not} the boundary of $B$.
Since the critical point $c$ in Definition \ref{def:crit_pts} is equidistant from all the points in $\cY$, we have that $c=C(\cY)$. Thus, we say that $c$ is the unique index $k$ critical point \textit{generated} by the $k+1$ points in $\cY$.
The last statement can be rephrased as follows:
\begin{lem} \label{lem:gen_crit_point}
A subset $\cY\subset \cP$ of $k+1$ points in general position generates an index $k$ critical point if, and only if, the following two conditions hold:
\begin{enumerate}[label=\bf{CP{\arabic*}}]
\item \label{cp1} $\quad C(\cY) \in \oconv(\cY)$,
\item \label{cp2} $\quad \cP \cap {B(\cY)}= \emptyset$.
\end{enumerate}
Furthermore, the critical point is $C(\cY)$ and the critical value is $R(\cY)$.
\end{lem}

Figure \ref{fig:crit_pts} depicts the generation of an index $2$ critical point in $\R^2$ by  subsets of $3$ points.
We shall also be interested in critical points $c$ that are within distance $\eps$ from $\cP$, i.e. $d_{\cP}(c)\le \eps$. This adds a third condition,
\begin{enumerate}[label=\bf{CP\arabic*}]
\setcounter{enumi}{2}
\item \label{cp3} $\quad R(\cY) \le \eps$.
\end{enumerate}
\begin{figure}[h]
\centering
 \includegraphics[scale=0.3]{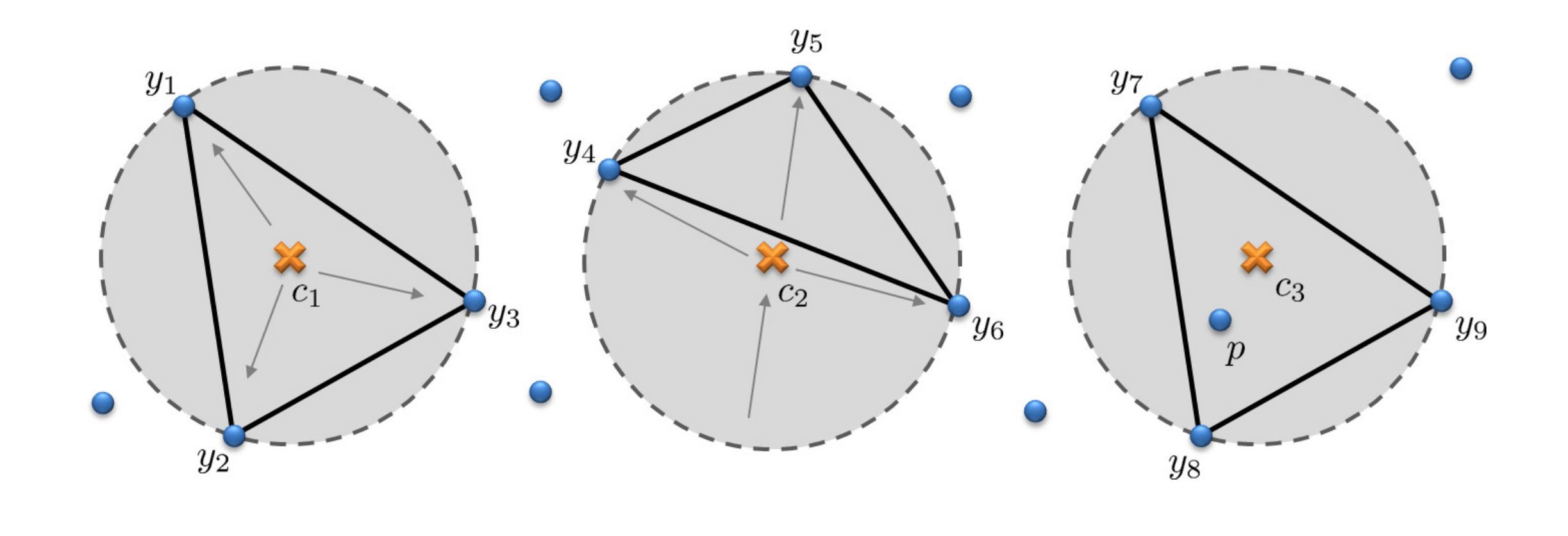}
\caption{\label{fig:crit_pts} Generating a critical point of index $2$ in $\R^2$ (i.e.\! a maximum point).
The small blue disks are the points of  $\cP$. We examine three subsets of $\cP$: $\cY_1 = \set{y_1,y_2,y_3}$, $\cY_2 = \set{y_4,y_5,y_6}$, and $\cY_3 = \set{y_7,y_8,y_9}$. $S(\cY_i)$ are the dashed circles, whose centers are $C(\cY_i) = c_i$. The shaded balls are $B(\cY_i)$, and the interior of the triangles are $\oconv(\cY_i)$.
(1) We see that both  $C(\cY_1) \in \oconv(\cY_1)$ \eqref{cp1} and $\cP\cap B(\cY_1) = \emptyset$ \eqref{cp2}. Hence $c_1$ is a critical point of index $2$.
(2) $C(\cY_2) \not\in \oconv(\cY_2)$, which means that \eqref{cp1} does not hold, and therefore $c_2$ is not a critical point (as can be observed from the flow arrows).
(3) $C(\cY_3) \in \oconv(\cY_3)$, so \eqref{cp1} holds. However, we have $\cP\cap B(\cY_3) = \set{p}$, so \eqref{cp2} does not hold, and therefore $c_3$ is also not a critical point. Note that in a small neighborhood of $c_3$ we have $d_{\cP} \equiv d_{\set{p}}$, completely ignoring the existence of $\cY_3$.}
\end{figure}

The following indicator functions, related to CP1--CP3, will appear often.
\begin{defn}\label{def:ind_fns}
Using the notation above,
\begin{align}
h(\cY) &\definedas \ind\set {C(\cY) \in \oconv(\cY)}\qquad\qquad\quad\ \ (\ref{cp1}) \label{eq:def_h} \\
h_\eps(\cY) &\definedas h(\cY) \ind_{[0,\eps]} (R(\cY)) \qquad\qquad \qquad\qquad (\ref{cp1}+\ref{cp3})\label{eq:def_h_eps}\\
g_\eps(\cY,\cP) &\definedas h_\eps(\cY) \ind\set{ \cP \cap {B(\cY)} = \emptyset} \ \qquad (\ref{cp1}+\ref{cp2}+\ref{cp3}) \label{eq:def_g_eps}
\end{align}
\end{defn}

\section{Main Results}
\label{sec:results}

Let $f$ be a bounded probability density function on $\R^d$  , which we assume to be bounded. This  assumption  will remain in force throughout the paper, without further comment.
Let $\cP_n$ be a spatial Poisson process on $\Rd$ with intensity function
$\lambda_n = n f$.
Denote by $\cC(n,k)$ the sets of critical points with index $k$ of $d_{\cP_n}$.
Let $\{r_n\}_{n=1}^{\infty}$ be a sequence of positive numbers with $\lim_{n\to\infty} r_n = 0$, and define
\beqq
N_{k,n} \definedas \#\{c\in \cC(n,k) \: d_{\cP_n}(c) \le r_n\}.
\eeqq

Our main goal  is to study the limits of $\Nk$ as $n\to\infty$. Since
 $ \E\{{N}_{0,n}\} \equiv n$ (the minima are the points of $\cP_n$)  we shall  only be interested in $1 \le k \le d$.
The results split into three main regimes, depending on the rate of convergence of $r_n$ to zero,  specifically, on the limit of the term $n\rnd$.

{A word on notation:} In the formulae presented below, for $g:(\R^d)^{k+1}\to\R$ and $\by= (y_1,\ldots,y_k)\in(\R^d)^k$ we write $g(0,\by)$  for $g(0,y_1,\ldots,y_k)$.
\subsection{The Subcritical Range ($n\rnd\to 0$)}\label{sec:results_subcrit}
This range is also known as the `dust phase', for reasons that will become clearer later, when we discuss  \cech complexes. We start with the limiting mean.
\begin{thm}[Limit mean]\label{thm:mean_subcrit}
If $n\rnd\to 0$, then for $1 \le k \le d$,
\beq
\label{limitmeanrates:eq}
\lim_{\ninf} (\factor)^{-1} \mean{\Nk}= \mu_k,
\eeq
where
\beq
\label{mu_k}
\mu_k = \frac{1}{(k+1)!}\int_{\Rd}f^{k+1}(x)dx \int_{(\Rd)^k} h_1(0,\by)d\by,
\eeq
and $h_1$ is defined in \eqref{eq:def_h_eps}.
\end{thm}
In general, as is common for results of this nature,
it is difficult to express this integral in a more transparent form.
However,  when $k=1$, $\by$ contains only a single point, and so
 $h\equiv 1$ and $R(0,\by) = \|\by\|/2$. Therefore,
$h_1(0,\by) = \ind\set{\norm{\by} \le 2}$, yielding $\mu_1 = 2^{d-1} \omega_d \int_{\Rd}f^2(x)\,dx$, where $\omega_d$ is the volume of the unit ball in $\R^d$.
Some numerics for other cases are given below.

The observation that, for a specific choice of $r_n$,
there is at most one $\alpha \in [1,d]$ such that $\limninf n^{\alpha+1} r_n^{d\alpha} \in (0,\infty)$ leads to the important fact that there is a
`critical' index, $ k_c \definedas \lfloor \alpha \rfloor $, such that
\beq
\qquad k < k_c \Rightarrow \limninf \mean{\Nk}  = \infty,  \label{kc:def}
\qquad k > k_c \Rightarrow\limninf \mean{\Nk}  = 0,  
\eeq
with any value in $(0,\infty]$ possible at $k=k_c$.
That is, there is phase transition occurring within the subcritical regime
itself. Similar regimes, with identical limits, appear
 for asymptotic variances.
\begin{thm}[Limit variance]\label{thm:var_subcrit}
If $n\rnd\to 0$, then for $1 \le k \le d$,
\[
\lim_{\ninf} (\factor)^{-1} \var{\Nk} = \mu_k.
\]
\end{thm}
Not surprisingly, the three regimes also yield different limit distributions.
\begin{thm}[Limit distribution]\label{thm:dist_subcrit}
Let $n\rnd\to 0$, and  $1 \le k \le d$,
\begin{enumerate}
\item  If $\lim_{n\to\infty} n^{k+1}r_n^{dk} = 0$, then
\[
\Nk \xrightarrow{L^2} 0.
\]
\item If $\lim_{n\to\infty} n^{k+1}r_n^{dk} = \alpha \in (0,\infty)$, then
\[
\Nk \xrightarrow{\cL} \pois{{\alpha \mu_k}}.
\]
\item If $\lim_{n\to\infty} n^{k+1}r_n^{dk} = \infty$, then
\[
\frac{\Nk - \mean{\Nk}}{({\factor})^{1/2}} \xrightarrow{\cL} \cN(0,\mu_k).
\]
\end{enumerate}
\end{thm}
As above, for a specific choice of $r_n$, there is going to be at most a single $k_c$ for which the Poisson limit applies.
Otherwise $\Nk$ converges either to zero or infinity. Thus,
in the subcritical regime, the picture is that
$n = N_{0,n} \gg N_{1,n} \gg \dots \gg N_{k_c,n}$, 
while, for $k>k_c$ the value of $\Nk$ will be zero, with high probability, which increases with $k$.

\subsection{The Critical and Supercritical Ranges ($n\rnd\to \lambda \in (0,\infty]$)}\label{sec:results_crit}

We now look at the critical ($n\rnd\to \lambda \in (0,\infty)$) and supercritical ($n\rnd\to \infty$) regimes. While there are differences between the two
 regimes, the general outline of the results is the same.
In both, the correct scaling for $\Nk$ is $n$ (as opposed to $\factor$ in the subcritical range). Consequently, the limit results are similar for all the indices.

The supercritical regime is significantly more difficult to analyze than
either the critical or subcritical, and we shall require an additional
assumption for this case, which necessitates a definition.
\begin{defn}\label{def:lower_bdd_dens}
Let $f:\R^d\to\R$ be a probability density function. We say that $f$ is \textit{lower bounded} if it has compact support and $\fmin \definedas\inf\set{f(x) \:x\in \supp(f)} > 0$.
\end{defn}
Henceforth, when dealing with the supercritical phase, we always assume that  $f$ is a lower bounded probability density, and that  $\supp(f)$ is  convex.
It is not clear at this point
 if these are necessary conditions, or a
consequence of our proofs.

\begin{thm}[Limit mean]\label{thm:mean_crit}
If $n\rnd\to \lambda\in(0,\infty]$, then, for $1 \le k \le d$,
\[
\lim_{\ninf} n^{-1} \mean{\Nk}= {\gamma_{k}(\lambda)},
\]
where
\begin{align*}
\gamma_k(\lambda) &=\frac{\lambda^k}{(k+1)!}  \int_{(\Rd)^{k+1}}f^{k+1}(x) h_1(0,\by) e^{-\lambda \omega_d R^d(0,\by)f(x)}\,  d\by dx ,\\
\gamma_k(\infty) &= \lim_{\lambda\to\infty} \gamma_k(\lambda) = \frac{1}{(k+1)!}\int_{(\R^d)^k} {h(0,\by) e^{-\omega_d R^d(0,\by)} \,d\by},
\end{align*}
$\omega_d$ is the volume of the unit ball in $\R^d$, and $R,\ h$, and $h_1$ are defined in \eqref{eq:def_R}, \eqref{eq:def_h},
and \eqref{eq:def_h_eps}, respectively.
\end{thm}
Again, these terms can  be evaluated for $k=1$, in which case
\beq
\gamma_1(\lambda) &=& \frac{\lambda}{2}  \int_{\R^d}\int_{\norm{y}_2\le 2}f^2(x) e^{-\lambda \omega_d 2^{-d} \norm{y}_2^d f(x)}\,dy dx, \label{eq:gamma_1_crit}\\
\gamma_1(\infty) &=& \frac{1}{2}\int_{\R^d} {e^{-\omega_d 2^{-d} \norm{y}_2^d}\, dy} \ =\ 2^{d-1}. \notag 
\eeq

For a uniform distribution on a compact set $D\subset \R^d$ it is easy to show that $\gamma_1(\lambda)$ is given by
\begin{equation}\label{eq:gamma_1_unif}
\gamma_1(\lambda)
=  2^{d-1}({1-e^{-\lambda \omega_d /\vol(D)}}),
\end{equation}
from which it is easy to check that  $\gamma_1(\lambda)\to\gamma_1(\infty)$ as $\lambda\to\infty$.
For higher indices, we have no analytic way to compute $\gamma_k(\lambda)$. However, it can be evaluated numerically, and an example is given in
 Figure \ref{fig:gamma_k} for the uniform distribution on $[0,1]^3$.
Note that, in that example, $\gamma_0(\infty) - \gamma_1(\infty) + \gamma_2(\infty) - \gamma_3(\infty) \approx 0$. This is not a coincidence, and the explanation for this phenomenon will be given in Section \ref{sec:random_cech}, where we discuss the mean Euler characteristic of \cech complexes.

\begin{figure}[h!]
\centering
  \includegraphics[scale=0.5]{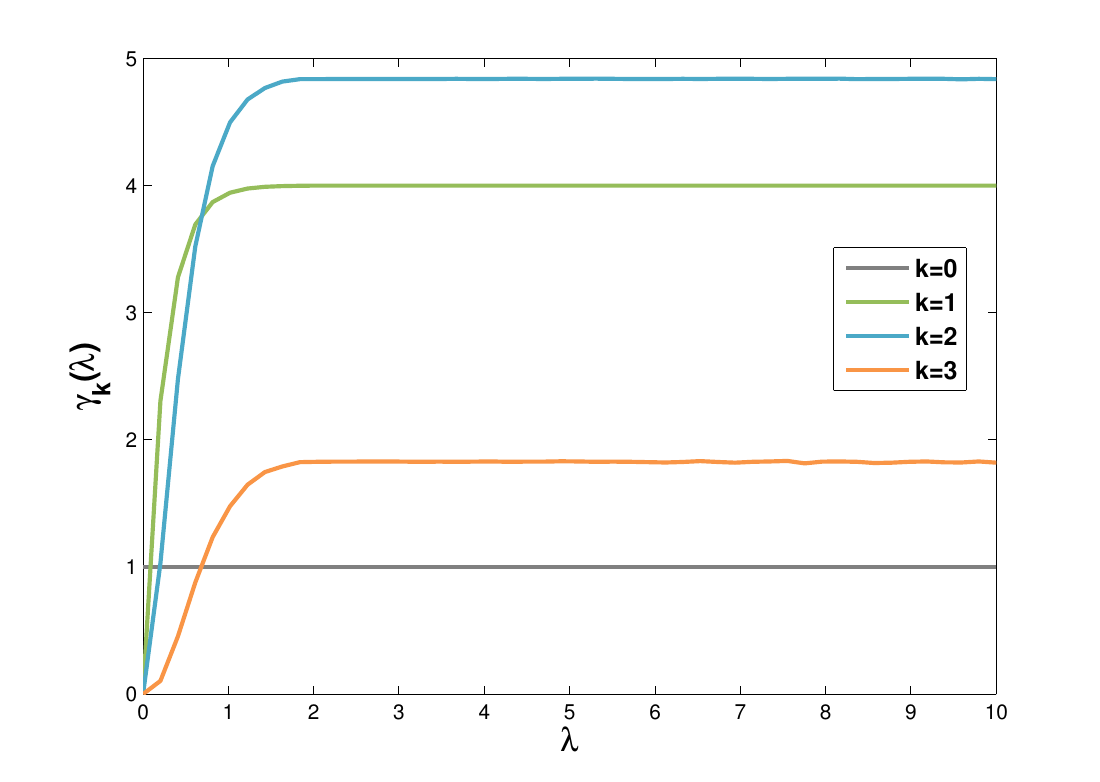}
\caption{\label{fig:gamma_k}The $\gamma_k(\lambda)$ function. In this example $d=3$, and $f(x)$ is the uniform density on $[0,1]^3$. For $k=0$ we know that $n^{-1}N_{0,n} = 1$, and for $k=1$ we have an explicit formula in \eqref{eq:gamma_1_unif}. For $k=2,3$ we used numerical integration followed by some smoothing. }
\end{figure}

Recall that, in the subcritical phase, the limit mean and the limit variance were exactly the same. For other phases, this is no longer true.
\begin{thm}[Limit variance]\label{thm:var_crit}
If $n\rnd\to \lambda\in(0,\infty]$ and  $1 \le k \le d$,
\beqq
\lim_{\ninf} n^{-1} \var{\Nk} = \sigma^2_k(\lambda)>0.
\eeqq
\end{thm}
 The expression defining $\sigma^2_k(\lambda)$ is rather complicated, and can be found in \eqref{sigma2khat:defn}.
Note, that as an immediate corollary of Theorems \ref{thm:mean_crit} and \ref{thm:var_crit}, we have the `law of large numbers' that,
 $   n^{-1}\Nk \xrightarrow{L^2} \gamma_k(\lambda).$

\begin{thm}[CLT]\label{thm:clt_crit}
If $n\rnd\to \lambda\in(0,\infty]$, then for $1 \le k \le d$,
\beq
\label{CLTa}
    \frac{\Nk - \E\Nk}{\sqrt{n}} \xrightarrow{\cL} \cN(0,\sigma^2_k(\lambda)),
\eeq
\end{thm}

To conclude this section, we note an interesting result which is unique to the supercritical regime, for which we define
$\Nkg \definedas \abs{\cC(n,k)}$,  the `global' number of critical points of the distance function $d_{\cX_n}$ in $\R^d$ (i.e. without requiring  \eqref{cp3}). We note first that $\Nk$ and  $\Nkg$ have identical asymptotic behaviors, at least
at the level of their first two moments and CLT:
\begin{thm}\label{thm:mean_global}
Let $f$ be lower bounded with a convex support.
Then, for $1 \le k \le d$,
\beqq
\lim_{\ninf} n^{-1} \mean{\Nkg} = \gamma_{k}(\infty),
\qquad
\lim_{\ninf} n^{-1} \var{\Nkg}= \sigma^2_{k}(\infty),
\eeqq
and
\beqq
\frac{\Nkg - \mean{\Nkg}}{\sqrt{n}} &\xrightarrow{\cL}& \cN(0,\sigma^2_k(\infty)).
\eeqq
\end{thm}
An obvious corollary of Theorem \ref{thm:mean_global}
is that $n^{-1}\mean{\Nkg-\Nk} \to 0$.
However, much more is true:
\begin{prop}\label{prp:global_vs_local}
Under the conditions of  Theorem \ref{thm:mean_global}, and if
 $n\rnd \ge D^\star \log n$, for sufficiently large ($f$-dependent)  $D^\star$,
 then, for $1 \le k \le d$,
\beq
\label{NkgclosetoNk}
\limninf \mean{\abs{\Nkg - \Nk}} = 0.
\eeq
\end{prop}
Thus, in the supercritical phase, the slow decrease of the radii $r_n$
implies that the global and the local number of critical points
are ultimately equal with high probability, despite the fact that
both grow to infinity with increasing $n$.
This is an interesting result, and will turn out to be important when we discuss the Euler characteristic of the \cech complex in the next section. The equality between the local and global counts can be explained if we study how well $\supp (f)$ is covered by the random balls of radius $r_n$. Denoting by $C_n$ the event that $\supp(f) \subset \bigcup_{X\in \cP} B_{r_n}(X)$, then similar methods as in \cite{hall1985coverage, janson1986random, abw2012} can be applied to show that if $n r_n^d \ge D\log n$, then $\prob{C_n}\to 1$. Thus, under the assumptions of Proposition \ref{prp:global_vs_local}, the support of $f$ is comppletely covered by the $r_n$-balls. Since all the critical points lie within the support, we have that they all should be accounted for in $\Nk$.
 Note that   \eqref{NkgclosetoNk} relies heavily  on the assumed convexity
of  $\supp(f)$. For example, take $f$ to be the uniform density on the
annulus $A = \set{x\in \R^2 : 1 \le \abs{x} \le 2}$. Then,
 for $n$ large enough, we would expect to have a maximum point (index 2) close to the origin. This critical point will be accounted for in $N_{2,n}^{(g)}$, but will be ignored by $N_{2,n}$, since its distance to $\cX_n$
is greater than $1$.  Thus, we would expect that $\E\{|{N_{2,n}^{(g)} - N_{2,n}}|\} \to 1$, which contradicts \eqref{NkgclosetoNk}

\section{Random \cech Complexes}

\label{sec:topology}

As mentioned already a number of times, the results of the previous section
regarding critical points of the distance function  have implications for
the homology and Betti numbers of certain random  \cech complexes,
and so are related to recent results of  \cite{kahle_random_2011} and
\cite{kahle_limit_2010}. Our plan in this section is to describe these
complexes and then the connections. We shall assume that the reader either
has a basic grounding in algebraic topology at the level of the first two chapters of
\cite{hatcher_algebraic_2002} or is prepared to accept a definition of
the $k$-th Betti number $\beta_k:=\beta_k(X)$ of a topological space
$X$ as the number of  $k$-dimensional `holes' in $X$, where a $k$-dimensional hole can be thought of as anything that can be continuously transformed into
 a  $k$-dimensional sphere. The zeroth Betti number,
$\beta_0(X)$, is merely the number of
connected components in $X$.

\subsection{\cech Complexes and the Distance Function}
The \cech complex generated by a set of points $\cP$
is a simplicial complex, made up of vertices, edges, triangles and higher dimensional faces. While its general definition is quite broad, we focus on
  the following special case.
\begin{defn}[\cech complex]\label{def:cech_complex}
Let $\cP = \set{x_1,x_2,\ldots}$ be a collection of points in $\R^d$, and let $\eps>0$. The \cech complex $\CC(\cP, \eps)$ is constructed as follows:
\begin{enumerate}
\item The $0$-simplices (vertices) are the points in $\cP$.
\item An $n$-simplex $[x_{i_0},\ldots,x_{i_n}]$ is in $\CC(\cP,\eps)$ if $\bigcap_{k=0}^{n} {B_{\eps}(x_{i_k})} \ne \emptyset$.
\end{enumerate}
\end{defn}
Figure \ref{fig:cech} depicts a simple example of a \cech complex in $\R^2$.
\begin{figure}[h!]
\centering
  \includegraphics[scale=0.4]{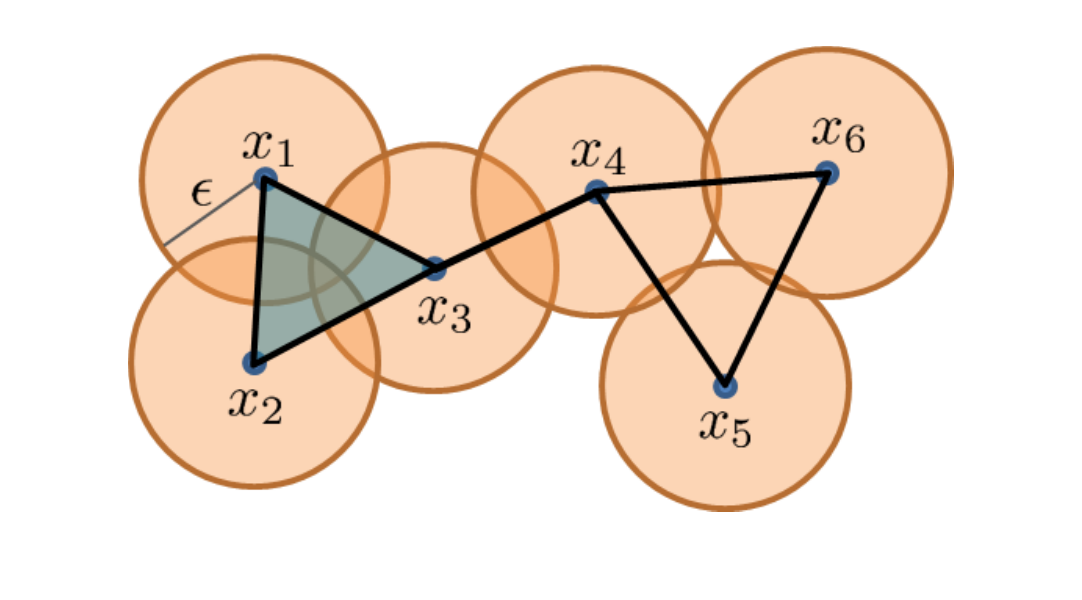}
\caption{\label{fig:cech} The \cech complex $\CC(\cP,\eps)$, for $\cP = \set{x_1,\ldots,x_6}\subset \R^2$, and some  $\eps$. The complex contains 6 vertices, 7 edges, and a single 2-dimensional face.}
\end{figure}
An important result, known as the `nerve theorem', links \cech complexes
 $\CC(\cP,\eps)$ and the neighborhood set  $\bigcup_{p\in \cP} B_{\eps}(p)$,
and states that they are  homotopy equivalent  (cf. \cite{borsuk_imbedding_1948}). Thus, for example,
they have the same Betti numbers. Furthermore, both are linked to
sublevel sets of the distance function, since it is immediate from the
definitions that
\beq
\label{distanceandcech}
\qquad
d_{\cP}^{-1} ([0,\epsilon])  = \set{x\in\R^d \: d_{\cP}(x) \le \eps}= \bigcup_{p\in \cP} B_{\eps}(p) \simeq \CC(\cP,\eps).
\eeq

\subsection{Critical Points and Betti Numbers}

Classical Morse theory, in particular the version developed in
 \cite{gershkovich_morse_1997} that applies to the distance function, tells us
that, in view of the equivalences in \eqref{distanceandcech}, there is a
connection between the critical points of
$d_\cP$ over the set $d_\cP^{-1} ([0,\epsilon])$, along with their indices, and the Betti numbers
of  $\CC(\cP,\eps)$. As usual, $\cP$ is a point set in $\R^d$, and assume that $\cP$ is in general position.
Then for every critical
point of $d_{\cP}$ at height $\epsilon$ and of index $k$,
for all small enough $\eta$, either
\beq \label{eq:morse_betti_inc}
\beta_k\left(\CC(\cP,\epsilon+\eta)\right) =
\beta_k\left(\CC(\cP,\epsilon-\eta)\right) +1,
\eeq
or\beq \label{eq:morse_betti_dec}
 \beta_{k-1}\left(\CC(\cP,\epsilon+\eta)\right) =
\beta_{k-1}\left(\CC(\cP,\epsilon-\eta)\right) - 1.
\eeq

Despite this connection, Betti numbers, dealing, as they do, with
`holes', are typically determined by  global phenomena, and this makes them
hard to study directly in the random setting. On the other hand,
the structure of critical points is a local phenomenon, which is why,
in the random case,  we
 can say more about  critical points than what is known for
  Betti numbers to date.

\subsection{Random \cech Complexes}\label{sec:random_cech}
Retaining the notation of the
previous section, and defining
$\bk  \definedas \beta_k(\CC(\cP_n,r_n))$,
our aim will be to examine relationships between the
random variables $\Nk$ and the $\bk$ and $\bkm$. In addition, we shall
 compare our results for
$\Nk$  to those of  \cite{kahle_random_2011} and \cite{kahle_limit_2010}
for $\bk$, using  Morse theory to explain the connections. Note that the results in \cite{kahle_random_2011} are phrased in terms of the random samples case (with $n$ $\iid$ points), however the proofs there can be easily adjusted to fit the Poisson case as well (as in \cite{kahle_limit_2010} or
\cite{bobrowski2012thesis}, where both the Poisson and random samples cases are treated).

In direct analogy to the results of Section  \ref{sec:results},
  \cite{kahle_random_2011,kahle_limit_2010} show that
 the limiting behavior of $\CC(\cP_n, r_n)$ splits into three main regimes, depending on the limit of $n\rnd$.
In the subcritical  ($n\rnd\to 0$) or dust phase, in which
 the \cech complex consists mostly of small disconnected particles and very few holes,  Theorem 3.2 in \cite{kahle_random_2011} states that for $ 1 \le k \le d-1$,
\[
\limninf (n^{k+2} r_n^{d({k+1})})^{-1} \mean{\bk} = D_k,
\]
for some constant $D_k$ defined in an integral form and related to
the  $\mu_k$  of our Theorem \ref{thm:mean_subcrit}.
In \cite{kahle_limit_2010} the subcritical phase is explored in more
detail, and limit theorems analogous to those of
 Theorem \ref{thm:dist_subcrit} are proved. Combining their results
 with those in Section \ref{sec:results_subcrit},  observe that the
 $\Nk$ and the $\beta_{k-1,n}$ exhibit similar limiting behavior, and are $O(\factor)$. Furthermore, 
 based on the expected values, we can informally summarize the relationship between the different $\Nk$ and $\bk$ as follows:
\beq \label{eq:Nk_diagram}
  \begin{array}{ccccccc}
    N_{1,n} & \gg & N_{2,n} & \gg & N_{3,n}& \gg  \cdots  \gg  & N_{k_c,n}\\
     &  & \approx &  & \approx & & \approx \\
    &  & \beta_{1,n} & \gg & \beta_{2,n} & \gg \cdots \gg & \beta_{k_c-1,n} ,
  \end{array}
\eeq
where by $a_n\approx b_n$ we mean that $a_n/b_n \to c\in(0,\infty)$ and by $a_n \gg b_n$ we mean that $a_n/b_n \to \infty$, and $k_c$ is as in
\eqref{kc:def}. For $k>k_c$ all terms are zero with high probability that grows with $k$.

Recall that Morse theory tells us that each critical point of index $k$
contributes either $+1$ to $\beta_{k,n}$ or $-1$ to $\beta_{k-1,n}$ (see \eqref{eq:morse_betti_inc},\eqref{eq:morse_betti_dec}). Splitting $\Nk$ accordingly as $\Nk = \Nk^+ + \Nk^-$, the diagram \eqref{eq:Nk_diagram} implies
 that $\Nk^- \gg \Nk^+$. In other words, most of the critical points of index $k$ destroy homology generators rather than create new ones.

For the other regimes, making statements about the \cech complex becomes extremely difficult, and thus the theory is still incomplete.

In the critical phase ($n\rnd\to \lambda\in(0,\infty)$), the \cech complex starts to connect and the topology becomes more complex. In addition, once $\lambda$ passes a certain threshold, a giant component emerges (cf.\ Chapter 9 of
\cite{penrose_random_2003}), from which comes the alternate description of this
phase as the
 `percolation phase'. Theorem 4.1 in \cite{kahle_random_2011} states that for $1 \le k \le d-1$,
\[
\limninf n^{-1} \mean{\bk} \in (0,\infty),
\]
although the exact limit is not computed. This agrees with the results in Section \ref{sec:results_crit} of this paper. The main difference between the
two
sets of results is that  for critical points we are able to give a closed
form expression for the limit mean of $\Nk$ (Theorem \ref{thm:mean_crit}), as well as stronger limit results (Theorems  \ref{thm:clt_crit}--\ref{prp:global_vs_local}).
 This will be useful below, when we discuss  Euler characteristics.

In the supercritical regime ($n\rnd\to\infty$) even less is known about the \cech complex. In general, the \cech complex becomes highly connected, the
topology becomes simpler and the Betti numbers decrease. Theorem 6.1 of
 \cite{kahle_random_2011} gives the precise results that if $f$ is a uniform density with a compact and convex support, and $\limninf(\log n / n)^{-1/d}r_n >0$ , then
\begin{equation}\label{eq:kahle_supercrit}
\limninf \prob{\beta_{0,n} = 1, \beta_{1,n} = \cdots = \beta_{d-1,n} = 0}  = 1,
\end{equation}
which is described in \cite{kahle_random_2011} by saying that $\CC(\cP_n,r_n)$ is ``asymptotically almost surely contractible''.
We have no analogous result about critical points, nor could we,
since $\Nk$ is  $O(n)$ and thus $\Nk \to \infty$ (Section \ref{sec:results_crit}). However, Corollary \ref{cor:cech_ec} below gives information about
the Euler characteristic of the \cech complex which is different from,
but related to, \eqref{eq:kahle_supercrit}. (Note  that
\eqref{eq:kahle_supercrit} requires that the underlying
 probability density  is lower bounded with  convex support,
the same assumption we adopted Section \ref{sec:results_crit}.)

To conclude this section, we present a novel statement about the \cech complex $\CC(\cP_n, r_n)$ which can be made based on the results in Section \ref{sec:results}. The Euler characteristic of a simplicial complex $\cS$ has a number
of equivalent definitions, and a number of important applications. One of
the definitions,  via Betti numbers, is
\begin{equation}\label{eq:ec_betti}
    \chi(\cS) = \sum_{k=0}^{\infty} (-1)^k \beta_k(\cS).
\end{equation}
However,  $\chi(\cS)$ also has a
definition via indices of critical points of appropriately defined functions
supported on $\cS$, and this leads to
\begin{cor}\label{cor:cech_ec}
Let $\chi_n$ be the Euler characteristic of $\CC(\cP_n, r_n)$. Then
\beq
\label{corollary:equ}
\qquad \limninf n^{-1} \mean{\chi_n}=
\begin{cases}
1 & n\rnd \to 0,\\
1+\sum_{k=1}^d {(-1)^k \gamma_k(\lambda) }&
n\rnd \to \lambda\in(0,\infty),\\
0 & n\rnd \to \infty.
 \end{cases}
\eeq
Moreover, when $n\rnd\to \infty$ and $n\rnd \ge D^\star \log n$ (with $D^\star$ as in Proposition \ref{prp:global_vs_local}), then $\mean{\chi_n}\to 1$.
\end{cor}
The proof of Corollary \ref{cor:cech_ec} is presented in Section \ref{sec:ec_proof}.
Note that \eqref{corollary:equ} cannot  be proven using only the
existing results
on  Betti numbers, since the values of the limiting mean in
the critical and supercritical regimes are not available.
This demonstrates one of the advantages of studying the homology of the
\cech complex via the distance function.

In closing we note some of the implications of Corollary \ref{cor:cech_ec}.
 In the subcritical phase, we have that $\chi_n \sim n$,
which agrees with the intuition developed so far
that, in this range, the \cech complex consists of mostly small disconnected particles and very few holes. In the critical range we have a non-trivial limit resulting from the fact that the \cech complex has many holes of all possible dimensions. In the supercritical range, $\chi_n \sim 1$ which is exactly what we get when $\beta_{0,n}=1,\beta_{1,n}=\cdots=\beta_{d-1,n} = 0$ (cf.\ \eqref{eq:ec_betti}, \eqref{eq:kahle_supercrit}). Since $n^{-1}\mean{\chi_n} \to 0$ in this regime, it is clear now why the numerics of Figure \ref{fig:gamma_k} showed
that $\sum_{k=0}^3 (-1)^k \gamma_k(\infty) \approx 0$. Finally, note that in a sequel \cite{bobrowski_topology_2014}, we explore the \cech complex when the samples are generated by a distribution supported on a closed manifold $\cM$. In this case we can make a much more concrete statement,  and prove that $\bk \to \beta_k(\cM)$ with an appropriate choice of  radius $r_n$. A different direction, in which the underlying samples are generated by stochastic processes with dependence between the points, can be found in \cite{YogiRobert}.

\section{Some Notation and Elementary Considerations}
The remaining sections of the paper are devoted to proofs of the results
in Sections \ref{sec:results}  and \ref{sec:topology}, and are organized
according to
situations: sub-critical (dust), critical (percolation), and super-critical. In this section
we list some common notation and note some simple facts that will be used
in many of them.

\begin{list}{\labelitemi}{\leftmargin=1em}
\item Henceforth, $k$ will be fixed, and whenever we use $\cY,\cY'$ or $\cY_i$ we implicitly assume that $\abs{\cY}=\abs{\cY'} = \abs{\cY_i} = k+1$, unless stated otherwise.
\item Usually, finite subsets of $\R^d$ will be denoted calligraphically ($\cX,\cY$). However inside integrals we  use boldfacing and lower case ($\bx,\by$).
\item For $x\in\R^d$, $\bx \in (\R^d)^{k+1}$ and $\by\in(\R^d)^{k}$, we use the shorthand
\begin{align*}
f(\bx) &\definedas f(x_1)f(x_2)\cdots f(x_{k+1}), \\
f(x+r_n \by) &\definedas f(x+r_n y_1) f(x+r_n y_2)\cdots f(x+r_n y_k),\\
h(0,\by) &\definedas h(0,y_1,\ldots,y_k).
\end{align*}
\item The symbol `$\const$', denotes a constant value, which might depend on $d$ (ambient dimension), $f$ (the probability density of the samples), and $k$ (the Morse index), but on neither $n$ nor $r_n$. The actual value of $\const$
may change between and even within lines.
\item While not exactly a notational issue, we shall often use the fact
that, for every $k$, $n^{-k}\binom{n}{k} \to 1/k!$ as $n\to\infty$, and
 there is a $\const$ such that $\binom{n}{k} \le \const n^k$.
\end{list}

\begin{lem}\label{lem:points_in_ball}
Let $\cX = (X_1,\ldots,X_k)$ be a set of $k$ i.i.d.\ points in $\Rd$ sampled
from  a bounded density $f$. Then there exists a constant $\const$ such that
\[
\prob{\cX \textrm{ is contained in a ball with radius $r$}} \le \const r^{d(k-1)}.
\]
\end{lem}
\begin{proof}
If $\cX$ is bounded by a ball with radius $r$, then $X_2,\ldots,X_k$ are all within distance $2r$ from $X_1$, thus
\begin{align*}
\prob{\cX \textrm{ is bounded by a ball of radius $r$}} &\le \int_{\Rd} \param{\int_{B_{2r}(x)} f(y)dy}^{k-1} f(x)dx \\
&\le \int_{\Rd} \param{\fmax \vol(B_{2r}(x)}^{k-1} f(x)dx \\
&= \fmax^{k-1}\omega_d^{k-1} (2r)^{d(k-1)}\\
&\definedas \const r^{d(k-1)},
\end{align*}
where $\fmax \definedas \sup_{x\in\R^d} f(x)$, and $\omega_d$ is the volume of the unit ball in $\R^d$.
\end{proof}

\section{Means for the Subcritical Range ($n\rnd\to 0$)}
We start by proving Theorem \ref{thm:mean_subcrit} (the limit expectation), which requires the following important lemma. Note that the lemma has two implications. Firstly, it gives a precise order of magnitude, with constant, for the
probability that $k+1$ points in the $r_n$-neighborhood of a point in $\cX_n$
generate an index-$k$ critical point. Secondly, it implies that if an
additional, high density set of Poisson points is added to the picture, the
probability that any of these will be in the ball containing the $k+1$
original points is of a smaller order of magnitude.

\begin{lem}\label{lem:lim_indicators}
Let $\cY \subset \cX_{n}$, be a set of $k+1$ random $\iid$ points with density function $f$, independent of the Poisson process $\cP_n$. Then,
\[
\limninf r_n^{-dk}\mean{\hrn(\cY)}  = \limninf r_n^{-dk}\mean{\grn(\cY,\cY\cup\cP_n)}   = (k+1)!\mu_k.
\]
\end{lem}
    \begin{proof}

Note that from the definition of $h_{\eps}(\cdot)$, it follows that
\[
h_{\eps}(x,x+\eps \by) \definedas h_{\eps}(x,x+\eps y_1,\ldots, x+\eps y_k) = h_1(0,\by).
\]
Thus, using the change of variables $\bx \to (x,x+r_n\by)$,
\begin{align}
\mean{\hrn(\cY)} &= \int_{(\R^d)^{k+1}} f(\bx) \hrn(\bx) d\bx \non \\
&= r_n^{dk} \int_{\R^d} \int_{(\R^d)^k} f(x)f(x+r_n\by) \hrn(x,x+r_n\by) d\by dx\non \\
&= r_n^{dk} \int_{\R^d} f(x)\int_{(\R^d)^k}f(x+r_n\by) h_1(0,\by) d\by dx.\label{eq:mean_hrn_integral}
\end{align}
Now, for $h_1(0,\by)$ to be nonzero, all the elements $y_1,\ldots,y_k\in \R^d$ must lie inside $B_2(0)$ - the ball of radius $2$ around the origin. Therefore,
\[
\abs{f(x+r_n \by) h_1(0,\by)} \le \fmax^k \ind_{B_2(0)}(y_1)\cdots \ind_{B_2(0)}(y_k),
\]
and applying the dominated convergence theorem (DCT) to \eqref{eq:mean_hrn_integral} yields
\begin{equation}
\limninf \int_{(\R^d)^k}f(x+r_n\by) h_1(0,\by) dxd\by = f^k(x) \int_{(\R^d)^k} h_1(0,\by)d\by,
\end{equation}
from which follows
\begin{equation} \label{eq:mean_hrn_lim}
\limninf r_n^{-dk}\mean{\hrn(\cY)} = \int_{\R^d} f^{k+1}(x) dx \int_{(\R^d)^k} h_1(0,\by)d\by = (k+1)!\mu_k,
\end{equation}
completing the proof for $\hrn(\cY)$.

Next, the definition of $\cP_n$ as a Poisson process with intensity $nf(x)$ implies
\[
\cmean{\grn(\cY, \cY\cup \cP_n)}{\cY} =\hrn(\cY) \cprob{B(\cY)\cap \cP_n = \emptyset}{\cY} = \hrn(\cY) e^{-np(\cY)}.
\]
Thus,
\begin{align*}
\mean{\grn(\cY,\cY\cup \cP_n)} &= \mean{\cmean{\grn(\cY,\cY\cup \cP_n)}{\cY}} \\
&=\int_{(\R^d)^{k+1}} f(\bx) \hrn(\bx) e^{-np(\bx)} d\bx\non\\
&=r_n^{dk} \int_{\R^d} f(x) \int_{(\R^d)^k} f(x+r_n\by) h_1(0,\by)e^{-np(x,x+r_n\by)}d\by dx,
\end{align*}

The integrand here is  smaller or equal to  the one in \eqref{eq:mean_hrn_integral}, therefore we can safely apply the DCT to it. To find the limit, first note that
\begin{align*}
np(x,x+r_n\by) &= n \int_{B(x,x+r_n\by)} f(z)dz \\
&=  n {\vol(B(x,x+r_n \by))} \frac{\int_{B(x,x+r_n\by)}f(z)dz }{\vol(B(x,x+r_n \by))} \\
&=  n \omega_d(r_n R(0,\by))^d \frac{\int_{B(x,x+r_n\by)}f(z)dz }{\vol(B(x,x+r_n \by))}.
\end{align*}
Applying the Lebesgue differentiation theorem yields
\[
\limninf \frac{\int_{B(x,x+r_n\by)}f(z)dz }{\vol(B(x,x+r_n \by))} = f(x).
\]
Therefore, since $n\rnd \to 0$, we have
\begin{equation}\label{eq:lim_np}
\limninf np(x,x+r_n\by) = 0.
\end{equation}
Thus, we have
\[
\limninf r_n^{-dk} \mean{\grn(\cY,\cP_n)} = \limninf r_n^{-dk} \mean{\hrn(\cY)} = (k+1)!\mu_k,
\]
and we are done.
\end{proof}

Using the previous lemma, it is now easy to prove Theorem \ref{thm:mean_subcrit}.
\begin{proof}[Proof of Theorem \ref{thm:mean_subcrit}]
Note  that
$
\Nk = \sum_{\cY \subset \cP_n} \grn(\cY, \cP_n)
$.
 Applying Theorem \ref{thm:prelim:palm} therefore yields that
\[
\mean{\Nk} = \frac{n^{k+1}}{(k+1)!} \mean{\grn(\cY',\cY' \cup \cP_n)},
\]
where $\cY'$ is a copy of $\cY$ independent of $\cP_n$.
 Lemma \ref{lem:lim_indicators} then implies
\[
\limninf (\factor)^{-1} \mean{\Nk} = \mu_k,
\]
as required.
\end{proof}

\section{Variances and Limit Distributions for the Subcritical Range}

The proofs of Theorems \ref{thm:var_subcrit} and \ref{thm:dist_subcrit}
 split into three different cases, depending on the limit of $\factor$.

\subsection*{\bf{Case 1:} $\factor \to 0$}
We start with the limit variance for this case.
\begin{proof}[Proof of Theorem \ref{thm:var_subcrit}]

We start by writing
\begin{align}
\mean{\Nk^2} &= \mean{\sum_{\cY_1\subset\cP_n}\sum_{\cY_2\subset\cP_n} {\grn(\cY_1,\cP_n)\grn(\cY_2, \cP_n)}}
\notag
\\
&=\sum_{j=0}^{k+1}\mean{\sum_{\cY_1\subset\cP_n}\sum_{\cY_2\subset\cP_n}{ \grn(\cY_1,\cP_n)\grn(\cY_2,\cP_n)}\ind\set{\abs{\cY_1\cap\cY_2}=j}} \notag
\\
&\definedas \sum_{j=0}^{k+1} \mean{{I}_j}. \label{eq:def_I_j}
\end{align}

Note that
\[
I_{k+1} = \sum_{ \cY_1 \subset \cP_n }
    \grn(\cY_1,\cP_n) =\Nk.
\]
Thus, from Theorem \ref{thm:mean_subcrit},
\begin{equation}\label{eq:lim_I_k+1}
\limninf (\factor)^{-1}\mean{{I}_{k+1}} = \mu_k.
\end{equation}

Next, for $0\le j<k+1$, using Corollary \ref{cor:prelim:palm2} we have
\begin{align*}
\mean{I_j} &= \const n^{2k+2-j}\mean{ \grn(\cY_1',\cY_{12}'\cup\cP_n)\grn(\cY_2',\cY_{12}'\cup\cP_n)}_{\abs{\cY_1'\cap\cY_2'}=j},
\end{align*}
where $\cY'_{12} = \cY'_1 \cup \cY'_2$ is a set of $2k-j$ $iid$ points in $\Rd$ with density $f(x)$, independent of $\cP_n$, and $\abs{\cY_1'\cap\cY_2'} = j$.

For $0 < j < k+1$, if $\abs{\cY_1\cap\cY_2}=j$ and $\grn(\cY_1',\cY_{12}'\cup\cP_n)\grn(\cY_2',\cY_{12}'\cup\cP_n)=1$, then necessarily the $2k+2-j$ points in $\cY'_1\cup\cY'_2$ are bounded by a ball of radius $2r_n$, and using Lemma \ref{lem:points_in_ball} we have
\[
\mean{I_j} = \le \const n^{2k+2-j} r_n^{d(2k+1-j)}.
\]
Thus,
\begin{equation}\label{eq:lim_I_j}
(\factor)^{-1}\mean{ I_j} \le \const(n\rnd)^{k+1-j}\to 0.
\end{equation}

For $j=0$, the sets $\cY'_1$ and $\cY'_2$ are independent. Since $\grn(\cY'_i,\cY'_{12}\cup\cP_n)\le \hrn(\cY'_i)$, we have
\[
\mean{ \grn(\cY'_1,\cY'_{12}\cup\cP_n)\grn(\cY'_2,\cY'_{12}\cup\cP_n)} \le \mean{ \hrn(\cY'_1)\hrn(\cY'_2)} = \param{\mean{\hrn(\cY'_1)}}^2.
\]
Therefore,
\[
\mean{I_{0}} \le \const n^{2(k+1)} \param{\mean{\hrn(\cY)}}^2.
\]
Using Lemma \ref{lem:lim_indicators} together with the fact that $\factor\to 0$ yields
\begin{equation}\label{eq:lim_I_0}
(\factor)^{-1} \mean{I_0} \le \const \factor \param{r_n^{-dk}\mean{\hrn(\cY)} }^2 \to 0.
\end{equation}

Combining \eqref{eq:lim_I_k+1}, \eqref{eq:lim_I_j}, and \eqref{eq:lim_I_0} yields
\[
\limninf (\factor)^{-1} \mean{\Nk^2} = {\mu_k}.
\]
In addition, Theorem \ref{thm:mean_subcrit} implies
\[
(\factor)^{-1}(\mean{\Nk})^2 = \factor\param{(\factor)^{-1}\mean{\Nk}}^2\to 0.
\]
Therefore, since $\text{var}\{\Nk\} = \E\Nk^2 - (\E\Nk)^2$, we conclude that
\[
\limninf (\factor)^{-1}\var{\Nk} = {\mu_k},
\]
which completes the proof.
\end{proof}


Next, we wish to prove the first part of Theorem \ref{thm:dist_subcrit}, i.e.\! that $\Nk\convinltwo 0$.

\begin{proof}[Proof of Theorem \ref{thm:dist_subcrit} - Part 1]
Clearly, it suffices to show that
\beq
\label{converges:eq}
\limninf \mean{\Nk^2} =  0.
\eeq
However, in the previous proof, we saw that
\[
\limninf (\factor)^{-1} \mean{\Nk^2} = {\mu_k}.
\]
Since $\factor\to 0$, \eqref{converges:eq} follows immediately, and we are
done.
\end{proof}
\subsection*{\bf{Case 2:} $\factor \to \alpha \in (0,\infty)$}

\begin{proof}[Proof of Theorem \ref{thm:var_subcrit}]
The proof in this case is similar to the previous one.
We define $I_j$ the same way as in \eqref{eq:def_I_j}.
The same arguments that led to \eqref{eq:lim_I_k+1} and \eqref{eq:lim_I_j} can be repeated here, providing the limits of $\mean{I_j}$, for $0< j \le k+1$ .
The only difference
is in how to bound the term $\mean{I_0}$.
For that, a proof in the spirit of Lemma \ref{lem:lim_indicators} can be used to show that
\[
\limninf r_n^{-2dk}\mean{ \grn(\cY'_1,\cY'_{12}\cup\cP_n)\grn(\cY'_2,\cY'_{12}\cup\cP_n)}_{\abs{\cY'_1\cap\cY'_2}=0}= ((k+1)!\mu_k)^2.
\]

Using Corollary \ref{cor:prelim:palm2}, we have
\begin{align*}
&\limninf(\factor)^{-1}\mean{{I}_0} \\
&\quad= \limninf (\factor)^{-1}\frac{n^{2k+2}}{((k+1)!)^2}\mean{ \grn(\cY'_1,\cY'_{12}\cup\cP_n)\grn(\cY'_2,\cY'_{12}\cup\cP_n)}_{\abs{\cY'_1\cap\cY'_2}=0},
\end{align*}
and therefore,
\[
\limninf(\factor)^{-1}\mean{{I}_0} =\alpha\mu_k^2.
\]

Finally, we also have
\[
\limninf (\factor)^{-1} \param{\mean{\Nk}}^2 = \alpha\mu_k^2.
\]
This completes the proof.

\end{proof}


For $X,Y$ random variables taking values in $\N$, define the total variation distance to be $d_{\mathrm{TV}}(X,Y) := \sup_{A\subset \N} \abs{\prob{X\in A} - \prob{Y\in A}}$.
To prove the Poisson limit of Theorem \ref{thm:dist_subcrit}, we need the following lemma.
\begin{lem}\label{lem:lim_pois_hrn}
Let $\Sk := \sum_{\cY\subset\cP_{n}} \hrn(\cY)$, and let $Z\sim\pois{\mean{\Sk}}$. If $n\rnd\to 0$, then
\[
\dtv{\Sk}{Z} \le \const n\rnd.
\]
\end{lem}
\begin{proof}
The proof is very similar to the proof of Theorem 3.4 in \cite{penrose_random_2003}, and uses the Poisson approximation given in Theorem \ref{thm:prelim:stein_bern}.
Let $A\subset \N$ be a set of natural numbers, we wish to bound the difference
\[
\Big|\prob{\Sk\in A} - \prob{Z\in A}\Big|.
\]
Start by conditioning on
 $\abs{\cP_n}$, the number of points in $\cP_n$.
\begin{equation}\label{eq:dtv_pois}
\begin{split}
&\Big|\prob{\Sk\in A} - \prob{Z\in A}\Big| \\
&\quad= \Big|\sum_{m=0}^\infty \param{\cprob{\Sk\in A}{\abs{\cP_n}=m} - \prob{{Z}\in A}}\prob{\abs{\cP_n}=m}\Big|\\
&\quad\le \sum_{m=0}^\infty \Big|\cprob{\Sk\in A}{\abs{\cP_n}=m} - \prob{{Z}\in A}\Big|\,\prob{\abs{\cP_n}=m}.
\end{split}
\end{equation}

Given $\abs{\cP_n} = m$, let $\cI_m = \set{\bi \subset \set{1,2,\ldots,m} \: \abs{\bi}=k+1}$.
Then, for $\bi = \set{i_0,\ldots,i_k}$, and $\cX_{\bi} = \set{X_{i_0},\ldots,X_{i_k}}$, we can write
\[
\Sk = \sum_{\bi\in \cI_m} \hrn(\cX_\bi).
\]

Set  $\cN_\bi = \set{\bj\in\cI_m \: \abs{\bi\cap\bj} > 0}$, and let $\sim$ be a relation on $\cI_m$ such that $\bi\sim\bj$ if and only if $\bj \in \cN_i$. For $\bi\ne \bj$, $\cX_\bi$ and $\cX_\bj$ are independent unless $\bj\in\cN_\bi$. Thus, the graph $(\cI_m, \sim)$ is the dependency graph for $\xi_\bi \definedas \hrn(\cX_\bi)$.

Now, if $\hrn(\cX_\bi)\ne 0$ then the $k+1$ points in $\cX_\bi$ are bounded by a ball of radius $r_n$, and using Lemma \ref{lem:points_in_ball} we have
\[
p_\bi \definedas \mean{\xi_\bi} \le \const r_n^{dk}.
\]
Therefore,
\begin{align*}
\sum_{\bi\in\cI_m}\sum_{\bj\in\cN_\bi} p_{\bi} p_{\bj} &\le \binom{m}{k+1}\param{\binom{m}{k+1}-\binom{m-k-1}{k+1}}\const r_n^{2dk} \le \const m^{2k+1} r_n^{2dk} .
\end{align*}
Next, if $\bi\sim\bj$ with $\abs{\bi\cap\bj} = l >0 $, and $\hrn(\cX_\bi)\hrn(\cX_\bj)\ne 0$, then necessarily the $2k+2-l$ points in $\cX_\bi \cup \cX_\bj$ are bounded by a ball of radius $2r_n$, and therefore,
\[
p_{\bi,\bj} \definedas \mean{\xi_\bi\xi_\bj} \le \const r_n^{d(2k+1-l)}.
\]
Thus,
\begin{align*}
\sum_{\bi\in\cI_m}\sum_{\bj\in\cN_\bi\backslash \set{\bi}} p_{\bi,\bj} &\le \sum_{l=1}^k \binom{m}{k+1} \binom{m-k-1}{k+1-l} \binom{k+1}{l}\const r_n^{d(2k+1-l)}\\
&\le \const\sum_{l=1}^k m^{2k+2-l} r_n^{d(2k+1-l)}.
\end{align*}

Finally, using Lemma \ref{lem:lim_indicators} it is easy to prove that
\[
\limninf (\factor)^{-1} \mean{\Sk} = {\mu_k},
\]
which implies that
\[
\frac{1}{\mean{\Sk}} \le \const(\factor)^{-1}.
\]

Therefore, from Theorem \ref{thm:prelim:stein_bern}, we can conclude that
\[
\abs{\cprob{\Sk\in A}{\abs{\cP_n}=m} - \prob{{Z}\in A}} \le \const n^{-(k+1)}\sum_{l=1}^k m^{2k+2-l} r_n^{d(k+1-l)}.
\]
Substituting back into \eqref{eq:dtv_pois}, we have
\[
\dtv{\Sk}{{Z}} \le  \const n^{-k+1} \sum_{l=1}^k r_n^{d(k+1-l)} \mean{\abs{\cP_n}^{2k+2-l}}.
\]
Since $|\cP_n|\sim\pois{n}$, it is easy to find a constant $\const$ such that
\[
\mean{\abs{\cP_n}^{2k+2-l}} \le \const  n^{2k+2-l},
\]
for every $1\le l \le k$. So, finally, we have that
\[
\dtv{\Sk}{{Z}} \le \const \sum_{l=1}^k  n^{k+1-l} r_n^{d(k+1-l)} \le \const n\rnd,
\]
since $n\rnd\to 0$ and so is bounded.

\end{proof}

Note that the previous result did not use on the assumption that $\factor\to\alpha \in(0,\infty)$. However, to prove an analogous result for $\Nk$ rather than $\Sk$ we shall need it. We shall also need the  following two lemmas, the second of which follows easily from the first, which itself follows from a simple calculation.

\begin{lem}\label{lem:dtv_bound}
Let $X,Y$ be integer random variables defined over the same probability space,  such that $\Delta \definedas X-Y \ge 0$.
Then $
\dtv{X}{Y} \le \mean{\Delta}.
$
\end{lem}


\begin{lem}\label{lem:dtv_bound_pois}
Let $X\sim\pois{\lambda_x},\ Y\sim\pois{\lambda_y}$.
Then $\dtv{X}{Y} \le \abs{\lambda_x-\lambda_y}$.
\end{lem}


\begin{proof}[Proof of Theorem \ref{thm:dist_subcrit} - Part 2]
For a start, we need to prove that $\dtv{\Nk}{\Sk} \le \const n\rnd$. To this
end,
 define  $\Delta \definedas \Sk-\Nk$ and note that $\Delta$ counts the number of subsets $\cY\subset\cP_{n}$ for which $\hrn(\cY)=1$ but $\grn(\cY,\cP_n)=0$. This implies that there exists $X\in \cP_n \bs \cY$ for which $X\in B(\cY)$.
Thus, $\Delta$ is bounded from above by $k+2$ times the number of $(k+2)$-subsets  contained in  a ball of radius $r_n$. From Lemma \ref{lem:points_in_ball} and Lemma \ref{lem:dtv_bound} we have
\[
\dtv{\Nk}{\Sk} \le \mean{\Delta} \le \const {n}^{k+2} r_n^{d(k+1)} \le \const(\factor)(n\rnd)\le \const(n\rnd),
\]
where we used the fact that $\factor$ is bounded.

Next,  if $Z_N\sim\pois{\mean{\Nk}}$ and $Z_S\sim\pois{\mean{\Sk}}$ , then from Lemma \ref{lem:lim_pois_hrn} and the triangle inequality,
\begin{align*}
\dtv{\Nk}{Z_N}&\le \dtv{\Nk}{\Sk} + \dtv{\Sk}{Z_S} + \dtv{Z_S}{Z_N} \\
&\le \const(n\rnd) +\dtv{Z_S}{Z_N}.
\end{align*}
Finally, Lemma \ref{lem:dtv_bound_pois} implies  that
\[
\dtv{Z_S}{Z_N} \le \abs{\mean{\Sk}-\mean{\Nk}} = \abs{\mean{\Delta}} \le \const(n\rnd).
\]
This completes the proof that $\dtv{\Nk}{Z_N} \le \const(n\rnd) \to 0$. From Theorem \ref{thm:mean_subcrit}, since $\factor\to \alpha$, we have that $\mean{\Nk} \to \alpha\mu_k$. Using the fact that $Z_N \sim \pois{\mean{\Nk}}$, it is easy to see that $\dtv{\Nk}{\pois{\alpha\mu_k}} \to 0$ which implies convergence in distribution.

\end{proof}

\subsection*{\bf{Case 3:} $\factor \to \infty$}

\begin{proof}[Proof of Theorem \ref{thm:var_subcrit} - Part 3 ($\Nk$ only)]
We start with the second moment of $\Nk$,

\begin{align*}
\mean{\Nk^2} &=  \mean{\sum_{
\cY_1\subset\cP_n  } \sum_{
\cY_2\subset\cP_n  }
\grn(\cY_1,\cP_n)\grn(\cY_2,\cP_n)} \\
&=\sum_{j=0}^{k+1}\mean{\sum_{
\cY_1\subset\cP_n  } \sum_{
\cY_2\subset\cP_n  }
    \grn(\cY_1,\cP_n)\grn(\cY_2,\cP_n)\ind\set{\abs{\cY_1\cap \cY_2}=j}} \\
    &:= \sum_{j=0}^{k+1} \mean{{I}_j}.
\end{align*}
As in the proof of the previous cases, we have that
\begin{align*}
\limninf (\factor)^{-1}\E{{I}_{k+1}} &=\mu_k,\quad
\limninf (\factor)^{-1}\E{{I}_j} &= 0, \  1 \le j \le k.
\end{align*}
However, in this case, $I_0$ requires a different treatment.
Recall that our interest is in the variance -  $\var{\Nk}$. So we have,
\begin{align*}
\var{\Nk^2} &= \mean{\Nk^2} - \param{\mean{\Nk}}^2 \\
&= \mean{{I}_{k+1}} + \sum_{j=1}^{k} \mean{{I}_j} + \param{\mean{\widehat{I}_0} - \param{\mean{\Nk}}^2}.
\end{align*}
Thus, to complete the proof, we need to show that
\[
\limninf (\factor)^{-1}\param{\mean{\widehat{I}_0} - \param{\mean{\Nkt}}^2} = 0.
\]
Applying Corollary \ref{cor:prelim:palm2} we have
\begin{align*}
\mean{{I}_0}= \param{\frac{n^{k+1}}{(k+1)!}}^2 \mean{\grn(\cY_1',\cY_{12}'\cup \cP_n)\grn(\cY_2',\cY_{12}'\cup \cP_n)}_{\cY_1'\cap\cY_2' =\emptyset},
\end{align*}
where $\cY_1'$ and $\cY_2'$ are sets of $\iid$ points with density $f$, independent of $\cP_n$, and $\cY_{12}' = \cY_1' \cup \cY_2'$.
Similarly, applying Theorem \ref{thm:prelim:palm}, we have
\begin{align*}
\mean{\Nk} &= \frac{n^{k+1}}{(k+1)!} \mean{\grn(\cY_1',\cY_1'\cup \cP_n)}.
\end{align*}
Therefore, we can write
\[
\param{\mean{\Nk}}^2 = \param{\frac{n^{k+1}}{(k+1)!}}^2 \mean{\grn(\cY_1',\cY_1'\cup \cP_n)\grn(\cY_2',\cY_2'\cup \cP_n')},
\]
where $\cP_n'$ is an independent copy of $\cP_n$.
Set
\[
\Delta \definedas {\grn(\cY_1',\cY_{12}' \cup \cP_n)\grn(\cY_2',\cY_{12}' \cup \cP_n)-\grn(\cY_1',\cY_1'\cup \cP_n)\grn(\cY_2',\cY_2'\cup \cP_n')}.
\]
Showing  that $n^{k+1} r_n^{-dk} \mean{\Delta} \to 0$ will complete the proof.
Set
\begin{align*}
\Delta_1 = \Delta \cdot\ind\set{B(\cY_1')\cap B(\cY_2') \ne \emptyset}, \qquad
\Delta_2 = \Delta\cdot\ind\set{B(\cY_1')\cap B(\cY_2') = \emptyset}.
\end{align*}
If $\Delta_1 \ne 0$ then all the elements in $\cY_1'$ and $\cY_2'$ are bounded by a ball of radius $2r_n$. Therefore, using Lemma \ref{lem:points_in_ball}
\begin{align*}
{\mean{{\Delta_{1}}}} \le  \const r_n^{d(2k+1)}.
\end{align*}
Next, note that
\begin{align*}
\Delta_2 &= \hrn(\cY_1')\hrn(\cY_2')\ind\set{B(\cY_1')\cap B(\cY_2')=\emptyset}\\
&\qquad \times \Big(\ind\set{\cP_n \cap B(\cY_1') = \emptyset}\ind\set{\cP_n \cap B(\cY_2') = \emptyset} \\
&\qquad\qquad\qquad - \ind\set{\cP_n \cap B(\cY_1') = \emptyset}\ind\set{\cP_n' \cap B(\cY_2') = \emptyset}\Big).
\end{align*}
If $\Delta_2 \ne 0$, then $B(\cY_1')$ and $B(\cY_2')$ are disjoint. Therefore, given $\cY_1'$ and  $\cY_2'$, the set $\cP_n\cap B(\cY_2')$ is independent of the set $\cP_n\cap B(\cY_1')$ (by the spatial independence of the Poisson process), and has the same distribution as $\cP_n'\cap B(\cY_2')$. Thus, $\cmean{\Delta_2}{\cY_1',\cY_2'} = 0$, which implies that $\mean{\Delta_2} = 0$.

To conclude, $\mean{\Delta} \le \const r_n^{d(2k+1)}$. Therefore,
\[
\limninf n^{k+1}r_n^{-dk} \mean{\Delta} \le \limninf \const (n\rnd)^{k+1} = 0.
\]
This completes the proof for the limit variance.
\end{proof}

Next, we wish to prove the CLT  in Theorem \ref{thm:dist_subcrit}.

\begin{proof}[Proof of Theorem \ref{thm:dist_subcrit} - Part 3 ]
The proof is based on the normal approximation for sums of dependent variables given by Stein's method (Appendix \ref{sec:apndx_stein}). We  start by counting only critical points located in a compact $A\subset \R^d$ for which $\int_A f(x)dx > 0$.
For a fixed $n$, let $\set{Q_{i,n}}_{i \in \N}$ be a partition of $\R^d$ into cubes of side $r_n$, and let $I_A\subset \N$ be the (finite) set of indices $i$ for which $Q_{i,n}\cap A \ne \emptyset$. For $i\in I_A$, set
\begin{equation}\label{eq:def_grn_i}
    \grn^{(i)}(\cY,\cP_n) \definedas \grn(\cY, \cP_n)\ind_{A\cap Q_{i,n}}(C(\cY)),
\end{equation}
where $C(\cY)$ is the critical point in $\R^d$ generated by $\cY$ (cf.\ \eqref{eq:def_C}).
That is, $\grn^{(i)}=1$ implies that $\cY$ generates a critical point located in $A\cap Q_{i,n}$.
Then
\[
\Nkts{(i)} \definedas \sum_{\cY\subset \cP_n} \grn^{(i)}(\cY,\cP_n),
\]
is the number of critical points inside  $A\cap\qin$, and
\[
    \Nkts{A} \definedas \#\set{\textrm{critical points of $d_{\cP_n}$ inside $A$}} = \sum_{i\in I_A} \Nkts{(i)}.
\]
First, as in  the proof of Theorem \ref{thm:var_subcrit}, one can show that
\begin{equation}\label{eq:var_NkA}
\mu_k(A) \definedas \limninf (\factor)^{-1} \var{\Nkts{A}} \in (0,\infty)
\end{equation}
Now, for $i,j\in I_A$, define the relation $i\sim j$ if the distance between $Q_{i,n}$ and $Q_{j,n}$ is less than $2r_n$.
Then $(I_A, \sim)$ is the dependency graph (cf.\ \eqref{def:dep_graph}) for the set $\set{\Nkts{(i)}}_{i\in I_A}$. This follows from the fact that
 a critical point located inside $\qin$ is generated by points of $\cP_n$ that are within distance $r_n$ from $\qin$ (along with the spatial independence of $\cP_n$). The degree of this graph is bounded by $5^d$. Consider the normalized random variables
\[
    \xi_i \definedas \frac{\Nkts{(i)}-\mean{\Nkts{(i)}}}{\param{\var{\Nkts{A}}}^{1/2}}.
\]
According to Theorem \ref{thm:clt_stein}, in order to prove a CLT for $\Nkts{A}$, all we have  to do now is to find bounds for $\mean{\abs{\xi_i}^p},\ p=3,4$ .

Let $B_{r_n}(\qin)\subset \R^d$ be the set of points within distance $r_n$ of $\qin$, and let $Z_i \definedas \abs{\cP_n \cap B_{r_n}(\qin)}$ be the number points of the Poisson process $\cP_n$ lying inside $B_{r_n}(\qin)$.
Then $Z_i \sim \pois{\lambda_i}$
where $\lambda_i = \int_{B_{r_n}(\qin)} nf(x)dx \le n \fmax (3r_n)^d$. Thus, $Z_i$ is stochastically dominated by a Poisson random variable with parameter $\const nr_n^d$. Now,
\[
\Nkts{(i)} \le \binom{Z_i}{k+1} \le \const Z_i^{k+1}.
\]
Therefore, for any $p\ge 1$,
\[
\mean{\abs{\Nkts{(i)}}^p} \le \const\mean{Z_i^{p(k+1)}} \le \const(n\rnd)^{p(k+1)} \le  \const(n\rnd)^{k+1},
\]
since $n\rnd$ is bounded (note that each of the $\const$'s stands for a different value).
Thus, it is easy to show that also
\[
\mean{\abs{\Nkts{(i)}-\mean{\Nkts{(i)}}}^p} \le \const(n\rnd)^{k+1}.
\]
Since $A$ is compact, there exists a constant $v$ such that $\abs{I_A} \le v r_n^{-d}$.
Therefore, for $p=3,4$,
\[
\sum_{i\in I_A} \mean{\abs{\xi_i}^p} \le \frac{vr_n^{-d} \const(n\rnd)^{k+1} }{\param{{\var{\Nkts{A}}}}^{p/2}}
=v \const(\factor)^{1-p/2} \param{\frac{(\factor)}{{\var{\Nkts{A}}}}}^{p/2}  \to 0,
\]
where we used the fact that $\factor \to \infty$ and the limit in Theorem \ref{thm:var_subcrit}. From Theorem \ref{thm:clt_stein}, we conclude that
\begin{equation}\label{eq:clt_A}
\frac{\Nkts{A}-\mean{\Nkts{A}}}{\param{\var{\Nkts{A}}}^{1/2}} \xrightarrow{\cL} \cN(0,1).
\end{equation}

Now that we have a CLT for $\Nkts{A}$, we need to extend it to one for
$\Nkt$. The method we shall use is exactly the same as the one used in \cite{penrose_random_2003}, but, for completeness, we nevertheless include it.

Set $A_M = [-M,M]^d$, $A^M = \R^d \backslash A_M$, and suppose that $M$ is large enough such that $\int_{A_M} f(z)dz > 0$. Set
\[
\zeta_n(A) = \frac{\Nkts{A} - \mean{\Nkts{A}}}{\param{\factor}^{1/2}} \qquad
\zeta_n = \frac{\Nk - \mean{\Nk}}{\param{\factor}^{1/2}}
\]
To complete the proof we need to show that $\abs{\prob{\zeta_n \le t} - \Phi(t/\sqrt{\mu_k})} \to 0$, where $\Phi(\cdot)$ is the standard normal distribution function. Clearly, $\zeta_n = \zeta_n(A_M) + \zeta_n(A^M)$, and from \eqref{eq:clt_A} we have that
\begin{equation}\label{eq:clt_AM}
\zeta_n(A_M) \xrightarrow{\cL} \cN(0,\mu_k(A_M)).
\end{equation}
For every $t\in \R$ and $M,\delta>0$ we have
\beq
 \notag
\abs{\prob{\zeta_n \le t} - \Phi(t/\sqrt{\mu_k})} &\le& \abs{\prob{\zeta_n \le t} - \prob{\zeta_n(A_M)\le t-\delta}}  \\ \notag
&&+\abs{\prob{\zeta_n(A_M) \le t-\delta} - \Phi((t-\delta)/\sqrt{\mu_k(A_M)})}\\
&&+\abs{\Phi\param{(t-\delta)/\sqrt{\mu_k(A_M)}} - \Phi\param{t/\sqrt{\mu_k}}}.
 \label{eq:clt_sum_3}
\eeq
Now,
\beqq
\prob{\zeta_n \le t} &=& \prob{\zeta_n(A_M) \le t-\delta ,\zeta_n \le t}  \
+\ \prob{\abs{\zeta_n(A_M)-t} < \delta ,\zeta_n \le t} \\
&&\qquad  \qquad \qquad  \qquad \qquad  \qquad
  +\ \prob{\zeta_n(A_M) \ge t+\delta ,\zeta_n \le t}.
\eeqq
Note that the first term equals $$\prob{\zeta_n(A_M)\le t-\delta} - \prob{\zeta_n (A_M) \le t-\delta, \zeta_n > t}.$$ Thus,
\begin{align*}
&\abs{\prob{\zeta_n \le t} - \prob{\zeta_n(A_M)\le t-\delta}} \le \prob{\zeta_n (A_M) \le t-\delta, \zeta_n > t} \\
&\quad +\prob{\abs{\zeta_n(A_M)-t} < \delta ,\zeta_n \le t} +\prob{\zeta_n(A_M) \ge t+\delta ,\zeta_n \le t} \\
&\quad\le \prob{\abs{\zeta_n(A^M)} > \delta} + \prob{\abs{\zeta_n(A_M)-t} < \delta}.
\end{align*}
From Chebyshev's inequality we have that  $\prob{\abs{\zeta_n(A^M)} > \delta} \le \delta^{-2}\var{\zeta_n(A^M)}$. From \eqref{eq:clt_AM}, we have that
\begin{align*}
\limninf \prob{\abs{\zeta_n(A_M)-t} < \delta} &= \Phi((t+\delta)/\sqrt{\mu_k(A_M)}) - \Phi((t-\delta)/\sqrt{\mu_k(A_M)})\\
&\le \frac{2\delta}{\sqrt{2\pi \mu_k(A_M)}}.
\end{align*}

Therefore,
\beqq
\limsup_{n\to\infty}\abs{\prob{\zeta_n \le t} - \prob{\zeta_n(A_M)\le t-\delta}} \le \frac{\mu_k(A^M)}{\delta^2} +  \frac{2\delta}{\sqrt{2\pi \mu_k(A_M)}}.
\eeqq
For $\eps >0$, choose $\delta = \epsilon \sqrt{\pi \mu_k} / 4$. Since $\lim_{M\to\infty} \mu_k(A_M) = \mu_k$, and \\ $\lim_{M\to\infty} \mu_k(A^M) =0$, there exists $M$ large enough such that $\mu_k(A_M) \ge \mu_k/2$, $\mu_k(A^M) \le \eps \delta^2/2$, and also $\abs{\Phi\param{(t-\delta)/\sqrt{\mu_k(A_M)}} - \Phi\param{t/\sqrt{\mu_k}}}<2\eps$.
 For this choice of $\delta,M$, using last displayed inequality, we have
\[
\limsup_{n\to\infty}\abs{\prob{\zeta_n \le t} - \prob{\zeta_n(A_M)\le t-\delta}} \le \eps.
\]
Finally, returning to \eqref{eq:clt_sum_3}, there exists $N>0$ such that for every $n>N$
\[
\abs{\prob{\zeta_n \le t} - \Phi(t/\sqrt{\mu_k})} < 4\eps.
\]
This completes the proof.
\end{proof}


\section{The Critical and Supercritical Ranges ($n\rnd\to \lambda \in (0,\infty]$)}\label{sec:results_crit_pfs}
We start with the expectation computations.
The following standard lemma is going to play a key role in the supercritical regime.
\begin{lem}\label{lem:convex_ball}
	Let $D\subset \R^d$ be a compact convex set with  positive Lebesgue measure, and let $B_r(x)\subset \R^d$ be the ball of radius $r$ around $x$.
	Then there exists a constant $\const$ such that for every $r<\diam(D)$ and $x \in D$,
\[
\vol(B_r(x) \cap D) \geq \const r^d.
\]
\end{lem}


The following Lemma is analogous to  Lemma \ref{lem:lim_indicators}.
\begin{lem}\label{lem:lim_indicators_crit}
Let $\cY$, be a set of $k+1$ $\iid$ random variables with density $f$, independent of the Poisson process $\cP_n$. Then,
\[
\limninf n^k\mean{\grn(\cY,\cY\cup\cP_n)}   = (k+1)!\gamma_k(\lambda).
\]
\end{lem}


\begin{proof}
Setting $s_n = n^{-1/d}$
and  mimicking the proof of Lemma \ref{lem:lim_indicators} we obtain
\begin{align}
&\mean{\grn(\cY,\cY\cup\cP_n)} = \int_{(\R^d)^{k+1}} f(\bx) \hrn(\bx) e^{-np(\bx)}d\bx \non \\
&\quad= s_n^{dk} \int_{\R^d} \int_{(\R^d)^k} f(x)f(x+s_n\by) \hrn(x,x+s_n\by)e^{-np(x,x+s_n\by)} d\by dx\non \\
&\quad= n^{-k} \int_{\R^d} f(x)\int_{(\R^d)^k}f(x+s_n\by) h_{\tau_n}(0,\by) e^{-np(x,x+s_n\by)} d\by dx,\label{eq:mean_grn_integral_crit}
\end{align}
where $\tau_n = r_n / s_n = n^{1/d}r_n$. We wish to apply the dominated convergence theorem for the last integral. Thus, we need to bound the integrand with an integrable expression.

In the critical range this is done much as in  the subcritical range. Since $n\rnd\to \lambda < \infty$, we have that $\tau_n$ is bounded by some value $M$.
Now, for $h_{\tau_n}(0,\by)$ to be nonzero, all the elements $y_1,\ldots,y_k\in \R^d$ must lie inside $B_{2\tau_n}(0) \subset B_{2M}(0)$. Therefore,
\[
\abs{f(x+s_n\by) h_{\tau_n}(0,\by) e^{-np(x,x+s_n\by)}} \le \fmax^k \ind_{B_{2M}(0)}(y_1)\cdots \ind_{B_{2M}(0)}(y_k),
\]
and this expression is integrable.

The last argument cannot be applied in the supercritical range, since then $\tau_n$ is no longer bounded. However, applying our additional, lower boundedness
 assumptions on the  $f$, we can proceed as follows. Since we now have $\fmin > 0$ we also have that
\begin{equation} \label{eq:prob_ball}
p(\bx) = \int_{B(\bx)}f(z)dz \, \ge\,  \fmin \vol(B(\bx)\cap \supp(f)).
\end{equation}

 If $h_{r_n}(\bx) \ne 0$, then necessarily $C(\bx) \in \oconv(\bx)$ and $R(\bx) \le r_n$ (cf.\ \eqref{eq:def_h_eps}). In addition, if $f(\bx) \ne 0$, then $\bx \subset \supp(f)$. Since we assume that $\supp(f)$ is convex, we have that $C(\bx)\in \supp(f)$ as well. Thus, $B(\bx)$ is a ball centered at $C(\bx)\in \supp(f)$, with radius $R(\bx)$ small enough, and  Lemma \ref{lem:convex_ball} yields
\[
\vol(B(\bx)\cap \supp(f)) \ge \const R^d(\bx).
\]
Using the inequality in \eqref{eq:prob_ball}, and the definition of $R(\bx)$ in \eqref{eq:def_R}, we have that
\[
	p(x,x+s_n\by) \ge \fmin c^\star R(x,x+s_n\by)  = \fmin c^\star s_n^d R^d(0, \by) = \fmin c^\star n^{-1} R^d(0, \by).
\]
This can be used to bound the integrand in \eqref{eq:mean_grn_integral_crit}, so that
\begin{equation}\label{eq:dct_supercrit}\begin{split}
\abs{f(x+s_n\by) h_{\tau_n}(0,\by) e^{-np(x,x+s_n\by)}} &\le \fmax^k e^{-n \fmin \const R^d(x,x+s_n\by)} \\&= \fmax^k e^{-\fmin\const R^d(0,\by)}.
\end{split}
\end{equation}
Next, note that for $i=1,\ldots,k$, $R(0,\by) \ge \norm{y_i}/2$. Thus,
\[
R^d(0,\by) \ge \frac{1}{{2^d}k}\sum_{j=1}^k  {\norm{y_j}^d},
\]
which implies that the expression in \eqref{eq:dct_supercrit} is indeed integrable, and so the DCT can be safely applied in both regimes.

Next, we compute the limit of the integral in \eqref{eq:mean_grn_integral_crit}. Note first  that
\begin{align*}
np(x,x+s_n\by) &= n \int_{B(x,x+s_n\by)} f(z)dz \\
&=  n {\vol(B(x,x+s_n \by))} \frac{\int_{B(x,x+s_n\by)}f(z)dz }{\vol(B(x,x+s_n \by))} \\
&=  n \omega_d(s_n R(0,\by))^d \frac{\int_{B(x,x+s_n\by)}f(z)dz }{\vol(B(x,x+s_n \by))}.\\
&=  \omega_d R^d(0,\by) \frac{\int_{B(x,x+s_n\by)}f(z)dz }{\vol(B(x,x+s_n \by))},
\end{align*}
and using the Lebesgue differentiation theorem yields
\[
\limninf np(x,x+s_n\by) = \omega_d R^d(0,\by) f(x).
\]
Taking the limit of all the other terms in \eqref{eq:mean_grn_integral_crit} we have
\[
\limninf n^k \mean{\grn(\cY,\cY\cup\cP_n)} = \int_{(\R^d)^{k+1}} f^{k+1}(x) h_{\tau_\infty}(0,\by) e^{-\omega_d R^d(0,\by) f(x)} d\by dx,
\]
where $\tau_\infty = \limninf \tau_n$. In the supercritical regime, $\tau_\infty = \infty$, and consequently $h_{\tau_\infty}(\cdot) = h_{\infty}(\cdot) =  h(\cdot)$. Thus,
\begin{align*}
\limninf n^k \mean{\grn(\cY,\cY\cup\cP_n)} &= \int_{(\R^d)^{k+1}} f^{k+1}(x) h(0,\by) e^{-\omega_d R^d(0,\by) f(x)} d\by dx\\  &= (k+1)!\gamma_k(\infty).
\end{align*}
In the critical range, $\tau_n \to \lambda^{1/d}$. Therefore,
\begin{align*}
&\limninf n^k \mean{\grn(\cY,\cY\cup\cP_n)} = \int_{(\R^d)^{k+1}} f^{k+1}(x) h_{\lambda^{1/d}}(0,\by) e^{-\omega_d R^d(0,\by) f(x)} d\by dx\\
&\quad= \lambda^k\int_{(\R^d)^{k+1}} f^{k+1}(x) h_{\lambda^{1/d}}(0,\lambda^{1/d}\bz) e^{-\lambda\omega_d R^d(0,\bz) f(x)} d\bz dx= (k+1)!\gamma_k(\lambda).
\end{align*}
This completes the proof.
\end{proof}


\subsection{Asymptotic Means}

Using Lemma \ref{lem:lim_indicators_crit} we can prove Theorem \ref{thm:mean_crit}.

\begin{proof}[Proof of Theorem \ref{thm:mean_crit}]
Using Theorem \ref{thm:prelim:palm},
\[
\mean{\Nk} = \frac{n^{k+1}}{(k+1)!}\mean{\grn(\cY',\cY' \cup \cP_n)},
\]
and, using Lemma \ref{lem:lim_indicators_crit},
\[
\limninf n^{-1} \mean{\Nk} = {\gamma_k(\lambda)},
\]
which completes the proof.
\end{proof}

\subsection{Asymptotic Variance}

\begin{proof}[Proof of Theorem \ref{thm:var_crit}]
As in the proof of Theorem \ref{thm:var_subcrit},
\beqq
\var{\Nk^2} = \mean{\Nk} + \sum_{j=1}^{k} \mean{I_j} + \param{\mean{I_0} - \param{\mean{\Nk}}^2},
\eeqq
where
\[
I_j = \sum_{
\cY_1\subset\cP_n  } \sum_{
\cY_2\subset\cP_n  }
    \grn(\cY_1,\cP_n)\grn(\cY_2,\cP_n)\ind\set{\abs{\cY_1\cap \cY_2}=j}\Big.
\]
From Corollary \ref{cor:prelim:palm2},
\[
\mean{I_j} =\frac{ n^{2k+2-j}}{j!((k+1-j)!)^2} \mean{\grn(\cY_1',\cY_{12}'\cup \cP_n)\grn(\cY_2',\cY_{12}'\cup \cP_n)}_{\abs{\cY_1'\cap\cY_2'} = j}.
\]
For $0 < j < k+1$, as in the proof of Lemma \ref{lem:lim_indicators_crit}, one can show that
\begin{align*}
&\limninf n^{d(2k+1-j)}\mean{\grn(\cY_1',\cY_{12}'\cup \cP_n)\grn(\cY_2',\cY_{12}'\cup \cP_n)}_{\abs{\cY_1'\cap\cY_2'} = j} \\
&\quad = \int\limits_{\R^{d(2k+2-j)}} f^{2k+2-j}(x) h_{\tau_\infty}(0,\by_1\cup\bz) h_{\tau_\infty}(0,\by_2\cup\bz) \\
&\qquad\qquad \qquad \qquad \times e^{-\vol(B(0,\by_1\cup\bz)\cup B(0,\by_2\cup\bz))f(x)}dx d\by_1 d\by_2 d\bz,
\end{align*}
where $x\in \R^d,\ \by_i \in \R^{d(k+1-j)}, \bz \in \R^{d(j-1)}$, and $\tau_\infty = \limninf n^{1/d} r_n$.
Therefore,
\[
\limninf n^{-1}\mean{I_j} = \gamma_k^{(j)}(\lambda),
\]
where
\begin{align*}
\gamma_k^{(j)}(\lambda) &\definedas \frac{\lambda^{2k+1-j}}{j!((k+1-j)!)^2}
\int\limits_{\R^{d(2k+2-j)}} f^{2k+2-j}(x) h_1(0,\by_1\cup\bz) h_1(0,\by_2\cup\bz) \\
&\qquad \qquad\qquad \qquad\qquad \times e^{-\lambda  \vol(B(0,\by_1\cup\bz)\cup B(0,\by_2\cup\bz)) f(x)}dx d\by_1 d\by_2 d\bz.
\end{align*}
for $\lambda\in(0,\infty)$, and
\begin{align*}
\gamma_k^{(j)}(\infty) &\definedas \frac{1}{j!((k+1-j)!)^2}
\int\limits_{\R^{d(2k+2-j)}} f^{2k+2-j}(x) h(0,\by_1\cup\bz) h(0,\by_2\cup\bz)\\
 &\qquad \qquad\qquad \qquad\qquad \times e^{-\vol(B(0,\by_1\cup\bz)\cup B(0,\by_2\cup\bz)) f(x)}dx d\by_1 d\by_2 d\bz.
\end{align*}
It is easy to show that $0 < \gamma_k^j(\lambda) <\infty$ for $\lambda\in(0,\infty]$.
For $j=0$, we define
\[
\Delta \definedas {\grn(\cY_1',\cY_{12}' \cup \cP_n)\grn(\cY_2',\cY_{12}' \cup \cP_n)-\grn(\cY_1',\cY_1'\cup \cP_n)\grn(\cY_2',\cY_2'\cup \cP_n')}
\]
so that
\[
\mean{I_0} - \param{\mean{\Nk}}^2 =  \frac{n^{2k+2}}{((k+1)!)^2} \mean{\Delta}.
\]
Now set
\beqq
\Delta_1 = \Delta \cdot \ind\set{B(\cY_1')\cap B(\cY_2') \ne \emptyset},
\qquad
\Delta_2 = \Delta\cdot\ind\set{B(\cY_1')\cap B(\cY_2') = \emptyset}.
\eeqq
Then, as in the proof of Theorem  \ref{thm:var_subcrit}, we can show  that
$\mean{\Delta_2} = 0$,
and
\begin{align*}
&\limninf n^{2k+1}\mean{\Delta_1}\\
&= \int_{\R^{d(2k+2)}} f^{2k+2}(x) h_{\tau_\infty}(0,\by_1) h_{\tau_\infty}(0,\by_2)\ind\set{B(0,\by_1)\cap B(z,z+\by_2) \ne \emptyset} \\
&\ \times\param{ e^{-\vol(B(0,\by_1)\cup B(z,z+\by_2)) f(x)}- e^{-\omega_d(R^d(0,\by_1)+R^d(0,\by_2))f(x)}}dx dz d\by_1 d\by_2 ,
\end{align*}
where $x,z\in\R^d$, and $\by_i \in (\R^d)^k$.
Thus,
\[
\limninf n^{-1}\param{ \mean{I_0} - \param{\mean{\Nk}}^2} = \gamma_k^{(0)}(\lambda),
\]
where
\begin{equation*}
\begin{split}
\gamma_k^{(0)}(\lambda) &\definedas \frac{\lambda^{2k+1}}{((k+1)!)^2}\\
&\times \int_{\R^{d(2k+2)}} f^{2k+2}(x) h_1(0,\by_1) h_1(0,\by_2)\ind\set{B(0,\by_1)\cap B(z,z+\by_2) \ne \emptyset} \\
&\times\param{ e^{-\lambda \vol(B(0,\by_1)\cup B(z,z+\by_2)) f(x)}- e^{-\lambda\omega_d(R^d(0,\by_1)+R^d(0,\by_2))f(x)}}dx dz d\by_1 d\by_2 ,
\end{split}
\end{equation*}
for $\lambda < \infty$, and
\begin{equation*}
\begin{split}
\gamma_k^{(0)}(\infty) &\definedas \frac{1}{((k+1)!)^2}\\
&\times \int_{\R^{d(2k+2)}} f^{2k+2}(x) h(0,\by_1) h(0,\by_2)\ind\set{B(0,\by_1)\cap B(z,z+\by_2) \ne \emptyset} \\
&\times\param{ e^{- \vol(B(0,\by_1)\cup B(z,z+\by_2)) f(x)}- e^{-\omega_d(R^d(0,\by_1)+R^d(0,\by_2))f(x)}}dx dz d\by_1 d\by_2.
\end{split}
\end{equation*}
To conclude, we have proven that
\beq
\label{sigma2khat:defn}
\limninf n^{-1} \var{\Nk} = \gamma_k(\lambda) + \sum_{j=0}^k \gamma_k^{(j)}(\lambda) \definedas {\sigma}^2_k(\lambda) \in (0,\infty),
\eeq
as required.
\end{proof}

\subsection{CLT}

Next, we prove the CLT result in Theorem \ref{thm:clt_crit},
again  using Stein's method, as in  the proof of Theorem \ref{thm:dist_subcrit}.

\begin{proof}[Proof of Theorem \ref{thm:clt_crit}]
We start again by counting only critical points located in a compact set $A\subset \R^d$, with $\int_A f(x)dx > 0$.

We define $Q_{i,n}, \Nkts{(i)},\Nkts{A}, \grn^{(i)},(I_A, \sim)$ and $\xi_i$ the same way as in the proof of Theorem \ref{thm:dist_subcrit}. Then, as in the proof of  Theorem \ref{thm:var_crit}, one can show that
\begin{equation}\label{eq:var_NkA}
\limninf n^{-1} \var{\Nkts{A}} \in (0,\infty).
\end{equation}
According to Theorem \ref{thm:clt_stein}, in order to prove a CLT for $\Nkts{A}$,  we need  to find bounds for $\mean{\abs{\xi_i}^p},\ p=3,4$ . We start with $p=3$.
\[
\mean{\param{\Nkts{(i)} - \mean{\Nkts{(i)}}}^3} = \sum_{j=0}^3\binom{3}{j}(-1)^j \param{\mean{\Nkts{(i)}}}^{3-j}\mean{\param{\Nkts{(i)}}^j}.
\]
The computation of the bound here is similar in spirit to the ones we used in the proof of Theorem \ref{thm:var_subcrit}, but technically more complicated,
 and we shall not give details. Rather, we   shall  suffice with
a brief description of  the main ideas: Every element in the sum can be expressed as the expectation of a triple sum of the form
\begin{equation}\label{eq:quadruple_sum}
\E\Bigg\{\sum_{\cY_1\subset \cP_n^{(1)}}\sum_{\cY_2\subset \cP_n^{(2)}}\sum_{\cY_3 \subset \cP_n^{(3)}} \grn^{(i)}(\cY_1, \cP_n^{(1)})\grn^{(i)}(\cY_2, \cP_n^{(2)})\grn^{(i)}(\cY_3, \cP_n^{(3)})\Bigg\},
\end{equation}
where each of the Poisson processes can either be equal to one of the others or an independent copy, depending on $j$. As for $\mean{\Delta_2}$ in the proof of Theorem \ref{thm:var_crit}, we can use Palm theory, collect all the terms in which at least one of the balls $B(\cY_i)$ is disjoint from the others, and show that they cancel each other. For each of the remaining terms, we can show that if $\abs{\cY_1\cup\cY_2\cup\cY_3} = 3k+3-j$, with $0 \le j \le  3k+3$, then
the relevant part of the sum in \eqref{eq:quadruple_sum} is bounded by  $\const n^{3k+3-j} s_n^{d(3k+2-j)} \rnd = \const n\rnd$. This bound is achieved using integral evaluations similar to the ones used in the proof of Theorem \ref{thm:var_crit}, along with the fact that all the points are located within distance of $r_n$ from the cube $Q_{i,n}$. Thus, we have
\[
\mean{\param{\Nkts{(i)} - \mean{\Nkts{(i)}}}^3} \le \const n \rnd.
\]
Recall, that $|I_A| \le \const r_n^{-d}$. Therefore,
\[
\sum_{i\in I_A}\mean{\abs{\xi_i}^3}  \le   \frac{\const  r_n^{-d} n \rnd}{\param{\var{\Nkts{A}}}^{3/2}} =
  \frac{\const n}{n^{3/2} \param{n^{-1}\var{\Nkts{A}}}^{3/2}}  \to 0.
\]
The proof for $p=4$ is similar. Thus, from Theorem \ref{thm:clt_stein} we have that
\[
\frac{\Nkts{A} - \mean{\Nkts{A}}}{\param{\var{\Nkts{A}}}^{1/2}} \xrightarrow{\cL} \cN(0,1).
\]
To conclude the proof, we need to show that the CLT for $\Nkts{A}$ implies a CLT for $\Nk$. This is done  exactly as for Part 3 of Theorem \ref{thm:dist_subcrit}.

\end{proof}

The only remaining results  in  Section  \ref{sec:results} that still require
proofs relate to the global number of critical points.

\begin{proof}[Proof of Theorem \ref{thm:mean_global}]
This theorem is proved exactly the same way as Theorems \ref{thm:mean_crit}, \ref{thm:var_crit}, and \ref{thm:clt_crit} are proved in the super-critical phase. The only difference is that, throughout, $h(\bx)$ replaces $h_{\tau_n}(\bx)$. This, however does not affect any of the results, since in the limit $h_{\tau_n}(\bx) \to h(\bx)$.
\end{proof}

\begin{proof}[Proof of Proposition \ref{prp:global_vs_local}]
The expected difference between the global and local number of critical points is given by
\begin{align}
&\mean{\Nkg - \Nk}\notag\\
& = \frac{n^{k+1}}{(k+1)!}n^{-k} \int\limits_{(\R^d)^{k+1}} f(x)f(x+s_n\by)(h(0,\by) - h_{\tau_n}(0,\by)) e^{-np(x,x+s_n\by)} d\by dx \notag\\
& =\frac{1}{(k+1)!} \int\limits_{(\R^d)^{k+1}} f(x)f(x+s_n\by)(h(0,\by) - h_{\tau_n}(0,\by)) n e^{-np(x,x+s_n\by)} d\by dx. \label{eq:diff_Nkg}
\end{align}
As in the proof of Theorem \ref{thm:mean_crit} (cf.\ \eqref{eq:dct_supercrit}), we can show that the integrand is bounded by
\begin{eqnarray}\label{eq:dct_supercrit_global}
f(x) \fmax^k (h(0,\by) - h_{\tau_n}(0,\by)) n e^{-\fmin\const R^d(0,\by)}.
\end{eqnarray}
Now note that if the integrand is nonzero then $h\ne h_{\tau_n}$, and so $R(0,\by) > \tau_n$. Therefore, $R^d(0,\by) > 1/2(R^d(0,\by) + n\rnd)$, and \eqref{eq:dct_supercrit_global} can be replaced by
\begin{equation}\label{eq:bound_Nkg}
f(x)\fmax^k (h(0,\by) - h_{\tau_n}(0,\by)) e^{-\fmin\const R^d(0,\by)/2}n e^{-\fmin\const n\rnd/2}.
\end{equation}
Assuming that $n\rnd \ge D^\star \log n$, with
$D^\star = (\fmin \const / 2)^{-1}$
then
$n e^{-\fmin\const n\rnd/2} \le 1$ and we obtain an integrable bound for the integrand. Thus, we can apply the DCT to \eqref{eq:diff_Nkg}. Finally, note that the bound we found in \eqref{eq:bound_Nkg} converges to zero (since $h_{\tau_n} \to h$), so we are done.
\end{proof}

\section{Euler Characteristic Results}
\label{sec:ec_proof}
Finally, we prove Corollary \ref{cor:cech_ec}.

\begin{proof}[Proof of Corollary \ref{cor:cech_ec}]
First note that $\mean{N_{0,n}} = n$. Thus,
\[
\mean{\chi_n} = n + \sum_{k=1}^d (-1)^k \mean{\Nk}
\]
The first two cases of the theorem are now obvious consequences of Theorems \ref{thm:mean_subcrit} and \ref{thm:mean_crit}. For the third case, using Theorem \ref{thm:mean_global}, we have
\[
\limninf n^{-1} \chi_n = \limninf n^{-1}\sum_{k=0}^d  (-1)^k \Nkg.
\]
However, since $\Nkg$ counts all the critical points in $\R^d$,  Morse theory
implies
\[
\sum_{k=0}^d (-1)^k \Nkg = \chi(\R^d) = 1,
\]
and we can conclude that
$\limninf n^{-1}\chi_n = 0$.

If, in addition, $\rnd$ satisfies the conditions of Proposition \ref{prp:global_vs_local} (i.e. $n\rnd \ge D^\star \log n$), then
\beqq
0 = \limninf \sum_{k=0}^d (-1)^k\mean{\Nkg - \Nk} = 1 - \limninf\chi_n,
\eeqq
which implies that $\chi_n \to 1$.
\end{proof}

\appendix

\section{Convergence of Random Variables}\label{sec:lim_dist}
Probability theory uses a number of different notions of convergence. Below we define  the ones used in this paper.

Let $X_1,X_2,\ldots$ be a sequence of real valued random variables, with the cumulative distribution function of $X_n$ given by
\[
	F_n(x) = \prob{X_n \le x},
\]
and let $X$ be random variable with  cumulative distribution function $F$.
\begin{defn}
$X_n$ converges \emph{in distribution}, or \emph{in law} to $X$ (denoted by $X_n \xrightarrow{\cL} X$) if
\[
	\limninf F_n(x) = F(x)
\]
for every $x\in \R$ at which $F(x)$ is continuous. 
\end{defn}
This type of convergence is also sometimes referred to as `weak convergence'.

\begin{defn}
$X_n$ converges in $L^p$ to $X$ (denoted by $X_n \xrightarrow{L^p} X$) if
\[
	\mean{\abs{X_n-X}^p} \to 0.
\]
\end{defn}

Finally, let $A$ be a Borel subset of $\R$, and define the probability measures $\mu_n$ and $\mu$ by
\[
	\mu_n(A) = \prob{X_n \in A}, \quad \mu(A) = \prob{X\in A}.
\]
Then the total variation distance between $X_n$ and $X$, or between $\mu_n$ and $\mu$, is defined as
\[
d_{\mathrm{TV}}(X_n,X) \equiv d_{\mathrm{TV}}(\mu_n,\mu) := \sup_{A} \abs{\prob{X_n\in A} - \prob{X\in A}}.
\]
where the supremum is taken over all Borel subsets of $\R$. This distance provides us with the following notion of convergence.
\begin{defn}
$X_n$ converges in the total variation distance ($ X_n \xrightarrow{\mathrm{TV}} X$) if
\[
	\limninf d_{\mathrm{TV}}(X_n,X) = 0,
\]
\end{defn}

Note that both $L^p$ and total variation convergence are stronger than convergence in distribution.  Further, while convergence in total variation and in distribution actually refer only convergence of (deterministic) measures and/or cdf's, convergence in $L^p$ demands that all the random variables involved are defined on a common probability space, and that the convergence is that of the  random variables themselves.

\section{Palm Theory for Poisson Processes}

This appendix contains a collection of definitions and theorems which are used in the proofs of this paper. Most of the results are cited from \cite{penrose_random_2003}, although they may not necessarily have originated there. However, for notational reasons we refer the reader to \cite{penrose_random_2003}, while other resources include \cite{stoyan_stochastic_1987, arratia_two_1989}.
The following theorem is very useful when computing expectations related to Poisson processes.

\begin{thm}[Palm theory for Poisson processes, \cite{penrose_random_2003}
Theorem 1.6]
\label{thm:prelim:palm}
Let $f$ be a probability density on $\R^d$, and let $\cP_n$ be a Poisson process on $\R^d$ with intensity $\lambda_n = n f$.
Let $h(\cY,\cX)$ be a measurable function defined for all finite subsets $\cY \subset \cX \subset \R^d$  with $\abs{\cY} = k$. Then
\[
    \E\Big\{\sum_{ \cY \subset \cP_n}
    h(\cY,\cP_n)\Big\} = \frac{n^k}{k!} \mean{h(\cY',\cY' \cup \cP_n)}
\]
where $\cY'$ is a set of $k$ $iid$ points in $\Rd$ with density $f$, independent of $\cP_n$.
\end{thm}
We shall also need the following corollary, which treats second moments:
\begin{cor}\label{cor:prelim:palm2}
With the notation above, assuming $\abs{\cY_1} = \abs{\cY_2} = k$,
\[
    \E\Big\{\sum_{ \substack {
                    \cY_1 ,\cY_2\subset \cP_n  \\
                    \abs{\cY_1 \cap \cY_2} = j }}
    h(\cY_1,\cP_n)h(\cY_2,\cP_n)\Big\} = {\frac{n^{2k-j}}{j!((k-j)!)^2}} \mean{h(\cY_1',\cY_{12}' \cup \cP_n)h(\cY_2',\cY_{12}' \cup \cP_n)}
\]
where $\cY'_{12} = \cY'_1 \cup \cY'_2$ is a set of $2k-j$ $iid$ points in $\Rd$ with density $f(x)$, independent of $\cP_n$, and $\abs{\cY_1'\cap\cY_2'} = j$.
\end{cor}

\begin{proof}
Given $\abs{\cP_n} = m$, the sum on the LHS is finite. Therefore,
\beq\label{eq:sum_1_cond}
&&\qquad \E\Big\{\sum_{ \substack { \cY_1 ,\cY_2\subset \cP_n  \\\abs{\cY_1 \cap \cY_2} = j }}h(\cY_1,\cP_n)h(\cY_2,\cP_n)}\Big|{\abs{\cP_n} = m\Big\}\\
&&\ \  = \binom{m}{2k-j}\binom{2k-j}{k}\binom{k}{j} \cmean{h(\cY_1,\cP_n)h(\cY_2,\cP_n)}{\abs{\cP_n} = m}_{\abs{\cY_1\cap\cY_2} = j}  \notag
\eeq
Choosing now all possible subsets $\cY$ of size $2k-j$, and splitting each of them into two arbitrary subsets $\cY_1,\cY_2$ of size $k$ with $\abs{\cY_1\cap\cY_2} = j$, yields
\beq
\label{eq:sum_2_cond}
&&\quad \E\Big\{\sum_{ \substack {\cY \subset \cP_n  \\\abs{\cY} = 2k-j }}h(\cY_1,\cP_n)h(\cY_2,\cP_n)}\Big|{\abs{\cP_n} = m\Big\}\\
&&\qquad = \binom{m}{2k-j} \cmean{h(\cY_1,\cP_n)h(\cY_2,\cP_n)}{\abs{\cP_n} = m}_{\abs{\cY_1\cap\cY_2} = j}.  \notag
\eeq
Combining \eqref{eq:sum_1_cond}, \eqref{eq:sum_2_cond}, and  Theorem \ref{thm:prelim:palm} for subsets $\cY$ of size $2k-j$ yields,
\begin{align*}
&\E\Big\{\sum_{ \substack { \cY_1 ,\cY_2\subset \cP_n  \\\abs{\cY_1 \cap \cY_2} = j }}h(\cY_1,\cP_n)h(\cY_2,\cP_n)\Big\}\\
&\qquad\qquad = \binom{2k-j}{k}\binom{k}{j}\E\Big\{\sum_{ \substack {\cY \subset \cP_n  \\\abs{\cY} = 2k-j }}h(\cY_1,\cP_n)h(\cY_2,\cP_n)\Big\} \\
&\qquad\qquad = \frac{n^{2k-j}}{j! ((k-j)!)^2} \mean{h(\cY_1',\cY_{12}'\cup\cP_n)h(\cY_2',\cY_{12}'\cup\cP_n)},
\end{align*}
where $\cY'_{12} = \cY'_1 \cup \cY'_2$ is a set of $2k-j$ $iid$ points in $\Rd$ with density $f(x)$, independent of $\cP_n$, and $\abs{\cY_1'\cap\cY_2'} = j$.
\end{proof}

\section{Stein's Method}\label{sec:apndx_stein}
In this paper we heavily used Stein's method to derive limit theorems for the sums of dependent Bernoulli variables. We need both the Poisson and normal approximations, which are presented below.

\begin{defn}\label{def:dep_graph}
Let $(I,E)$ be a graph. For $i,j\in I$ we denote $i \sim j$ if $(i,j) \in E$. Let $\set{\xi_i}_{i\in I}$ be a set of random variables. We say that $(I,\sim)$ is a dependency graph for $\set{\xi_i}$ if for every $I_1\cap I_2 = \emptyset$, with no edges between $I_1$ and $I_2$, the set of variables $\set{\xi_i}_{i\in I_1}$ is independent of $\set{\xi_i}_{i\in I_2}$. We also define the neighborhood of $i$ as $\cN_i := \set{i} \cup \set{j \in I \: j\sim i}$.
\end{defn}

\begin{thm}[Stein's Method for Bernoulli Variables, Theorem 2.1 in \cite{penrose_random_2003}] \label{thm:prelim:stein_bern}
Let $\set{\xi_i}_{i\in I}$ be a set of Bernoulli random variables, with dependency graph $(I,\sim)$. Let
\[
p_i \definedas \mean{\xi_i},  \ \ p_{i,j} \definedas \mean{\xi_i \xi_j}, \ \
 \lambda \definedas \sum_{i\in I} p_i, \ \ W\definedas \sum_{i\in I} \xi_i,
\ \ Z\sim\pois{\lambda}.
\]
Then,
\[
\dtv{W}{Z} \le \min(3,\lambda^{-1}) \Big(\sum_{i\in I}\sum_{j\in \cN_i\bs \set{i}} p_{ij} + \sum_{i\in I}\sum_{j \in \cN_i} p_i p_j\Big).
\]
\end{thm}
\begin{thm}[CLT for sums of weakly dependent variables, Theorem 2.4 in  \cite{penrose_random_2003}]\label{thm:clt_stein}
Let $(\xi_i)_{i\in I}$ be a finite collection of random variables, with $\mean{\xi_i} = 0$. Let $(I,\sim)$ be the dependency graph of $(\xi_i)_{i\in I}$, and assume that its maximal  degree is $D-1$. Set $W\definedas\sum_{i\in I} \xi_i$, and suppose that $\mean{W^2}=1$. Then for all $w\in \R$,
\[
\abs{F_W(w) - \Phi(w)} \le 2(2\pi)^{-1/4} \sqrt{D^2 \sum_{i\in I} \mean{\abs{\xi_i}^3}} + 6\sqrt{D^3 \sum_{i\in I} \mean{\abs{\xi_i}^4}},
\]
where $F_W$ is the  distribution function of $W$ and $\Phi$ that of a standard Gaussian.
\end{thm}

\bibliographystyle{plain}
\bibliography{phd}



\end{document}